\documentclass[10pt]{article}

\usepackage{amssymb}
\usepackage{amsmath}
\usepackage{theorem}
\usepackage{epsfig}
\usepackage{verbatim}
\usepackage{graphicx}
\usepackage{subfigure}
\usepackage{cancel}
\usepackage{epstopdf}
\usepackage{a4wide}
\usepackage{pdflscape}
\usepackage{color}

\newtheorem{thm}{Theorem}[section]
\newtheorem{cor}[thm]{Corollary}

\newtheorem{rem}[thm]{Remark}

\newtheorem{lemma}[thm]{Lemma}

\newcommand{\T}{\mathbb{T}}
\newcommand{\R}{\mathbb{R}}
\newcommand{\ep}{\varepsilon}
\newcommand{\pa}{\partial}

\newcommand{\p}[1]{\left(#1\right)}

\newcommand{\sign}{\text{sign}}

\numberwithin{equation}{section}

\allowdisplaybreaks[4]

\title{Singularity formation for the fractional Euler-Alignment system in 1D}
\author{Victor Arnaiz\footnote{victor.arnaiz@icmat.es}\hspace{0.1cm}  and  \'Angel Castro\footnote{angel\_castro@icmat.es}\\ \\
\small Instituto de Ciencias Matem\'aticas ICMAT-CSIC-UAM-UCM-UC3M\\ \small 28049, Madrid, Spain.}
\date{}

\newenvironment{proof}{\begin{trivlist} \item[] {\em Proof:}}{\hfill $\Box$
                       \end{trivlist}}

\makeatletter
\renewcommand*\l@section{\@dottedtocline{1}{0em}{1.5em}}
\renewcommand*\l@subsection{\@dottedtocline{2}{1.5em}{2.3em}}
\renewcommand*\l@subsubsection{\@dottedtocline{3}{3.8em}{3.7em}}
\makeatother

\begin{document}

\maketitle

\begin{abstract}
We study the formation of singularities for the Euler-Alignment system with influence function $\psi=\frac{k_\alpha}{|x|^\alpha}$ in 1D. As in \cite{C} the problem is reduced to the analysis of a nonlocal 1D equation. We show the existence of singularities in finite time for any $\alpha$ in the range $0<\alpha<2$ in both  the real line and the periodic case.
\end{abstract}

\section{Introduction}

The Euler-Alignment system for the density $u$ and the velocity $v$ in 1D is given by
\begin{align}\label{EA}
\pa_t\rho +(v u)_x=&0\nonumber\\
\pa_t v + vv_x=&\int_{\R}\psi(|x-y|)(v(x)-v(y))u(y)dy.
\end{align}
This system is the macroscopic version (see \cite{HT}) of the Cucker-Smale model \cite{CS},
\begin{align}\label{CS}
\frac{dx_i(t)}{dt}=&v_i(t)\nonumber\\
\frac{dv_i(t)}{dt}=&\frac{1}{N}\sum_{j=1}^N \psi(|x_i(t)-x_j(t)|)(v_i(t)-v_j(t)),
\end{align}
which models the behavior of a collection of agents. In \eqref{CS}, $(x_i,v_i)$ are the position and velocity of each agent, $N$ is the total number of agents and $\psi$ is the influence function which measures the strength of the velocity alignment between two agents. In this paper we will focus on the case in which
\begin{align}\label{influence}
\psi(|x|)=\frac{k_\alpha}{|x|^{1+\alpha}},&& k_\alpha=-\frac{2^\alpha\Gamma\p{\frac{1+\alpha}{2}}}{\pi^\frac{1}{2}\Gamma\p{-\frac{\alpha}{2}}},
\end{align}
for $0<\alpha<2$.

The Euler-Alignment system \eqref{EA}-\eqref{influence} was studied  in the periodic case independently by T. Do, A. Kiselev, L. Ryzhik and C. Tan in \cite{DKRC} and by R. Shvydkoy and E. Tadmor in \cite{ST1} and \cite{ST2}. In these papers the authors show global existence of solutions for positive initial density $u_0>0$ and $0<\alpha<2$. In addition, in \cite{C}, C. Tan  showed  the existence of initial data $(u_0, v_0)$, with $u_0\geq 0$, such that density solution of \eqref{EA}-\eqref{influence}, $u(x,t)$ has not uniformly bounded $C^1$-norm for all time.

This paper is concerned with the existence of singularities in finite time for \eqref{EA}-\eqref{influence}. We will use the same observation than in \cite{C} which comes from \cite{DKRC}: by defining $G=v_x-\Lambda^\alpha u$ then \eqref{EA}-\eqref{influence} can be written in the following form
\begin{align*}
\pa_t u+(v u)_x=&0,\\
\pa_tG+(vG)_x=&0,\\
v=&G+\Lambda^\alpha u.
\end{align*}
Thus, if initially $G(x,0)=0$, $G(x,t)$ must be zero for all time and the system is reduced to the equation

\begin{align}
\pa_t u +\left(u \Lambda^{\alpha-1}Hu\right)_x=&0\label{cht}\\
u(x,0)=& u_0(x)\label{di},
\end{align}
where $0<\alpha<2$. Here, for $0<\alpha<1$,
\begin{align}\label{Lambdadefi}
\Lambda^{\alpha-1}u(x)=c_\alpha\int_{\R}\frac{u(y)}{|x-y|^{\alpha}}dy,\quad c_\alpha\equiv \frac{\Gamma\left(\frac{\alpha}{2}\right)}{\sqrt{\pi}2^{1-\alpha}\Gamma\left(\frac{1-\alpha}{2}\right)},
\end{align}
and
\begin{align*}
Hu(x)=\frac{1}{\pi}\int_{-\infty}^\infty \frac{u(y)}{x-y}dy.
\end{align*}

Therefore if one can show the existence of a singularity for \eqref{cht} in finite time one actually shows a singularity for \eqref{EA}-\eqref{influence}.

In the real line case,  P. Biler, G. Karch and R. Monneau found  the existence of self-similar solutions of \eqref{cht} which are $C^{\frac{\alpha}{2}}(\R)$. Indeed, they gave an explicit formula for the profile of these self-similar solutions, $\phi(x)=K(\alpha)(1-x^2)_+^\frac{\alpha}{2}$, where $K(\alpha)$ is a suitable constant. Their motivation was the study of the dynamics of the dislocation in a solid.

We will consider the equation \eqref{cht} in both the real line $\R$ and the circle $\T$ ($2\pi-$periodic functions) and the goal is to prove the formation of singularities in finite time from a smooth initial data. In order to do it we will impose some of the following conditions on the initial data:

\begin{enumerate}
\item [H1] In the real line setting : $u_0$ is compactly supported and $u_0(x)\geq 0$.
\item [H2] $u_0(0)=u_{0\,x}(0)=0$.
\item [H3] $u_0(x)=u_0(-x)$.
\item [H4] In the periodic case: $u_{0\,x}(x)\geq 0$ for $x\in [0,\pi]$.
\end{enumerate}

The main results of this paper are the following:

\begin{thm}\label{thmline} Let $u_0\in C^{\alpha^+}\p{\R}$, $1<\alpha<2$, satisfying H1, H2 and H3 and let $$u(x,t)\in C([0,T); C^{1+\beta^+}(\R))\cap C^1([0,T); C^{0^+}(\R))$$ a solution of \eqref{cht}, \eqref{di}. Then there exist a time $T<\infty$, such that
\begin{align*}
\lim_{t \to T^-} ||u(\cdot,t)||_{C^{\alpha^+}}=\infty.
\end{align*}
\end{thm}

\begin{thm}\label{thmperiodic} Let $u_0\in C^{1}\p{\T}$, $0 <\alpha<1$,  satisfying  H2, H3 and H4, and let
$$u(x,t)\in C([0,T); C^{1+\beta^+}(\R))\cap C^1([0,T); C^{0^+}(\R))$$ a solution of \eqref{cht}, \eqref{di}.

Let $$\int_{\R}x^{-1-\alpha}u_0(x)dx$$ be large enough with respect to $||u_0||_{L^\infty}$ and $\frac{1}{\alpha}$.

Then there exists a time $T<\infty$ such that
\begin{align*}
\lim_{t \to T^-} ||u(\cdot,t)||_{C^{1}}=\infty.
\end{align*}
\end{thm}

The proof of theorem \ref{thmline} will be given in section \ref{line} and the prooof of theorem \ref{thmperiodic} in section \ref{periodic}.

We emphasis that theorem \ref{thmline} works in the range $1<\alpha<2$ and theorem \ref{thmperiodic} in the range $0<\alpha<1$. Let us explain why we work in different settings for different ranges of $\alpha$. Theorem \ref{thmline} could be proven in the range $0 <\alpha<1$ but removing the condition $H1$. Indeed we would have to impose the condition $u_{0\,x}\geq 0$ for $x\geq 0$, which means that $u_0$ does not vanish at infinity. Even one can guest that equation \eqref{cht} make sense in spaces including these kind of initial data, the need of  the lack of decay at infinity (infinite mass) is  something that one would like to avoid. Because of that we also prove the existence of singularities in the range $0<\alpha<1$ in the periodic case where the mass $\int_{\T} u_0 dx$ is finite. The existence of singularities  should be true in the range $1<\alpha<2$ in the periodic case without any condition on the monotonicity of the initial data nor on the size of the initial data. The proof would be given by a combination of the proofs of theorems \ref{thmline} and \ref{thmperiodic}. We do not include it here for the sake of simplicity.

\section{The case $\alpha=1$}\label{alpha1}

The case $\alpha=1$ have already been studied in different context. In  \cite{CCCF} D. Chae, A. C\'ordoba, D. C\'ordoba and M. A. Fontelos introduced equation \ref{cht} as 1D model of the Surface Quasi-Geostrophic equation. They proved the existence of singularities for some initial data with non zero negative part. In \cite{CC} D. C\'ordoba and the second author showed global existence and gain of analyticity for $u_0>0$, ill-posedness  in $H^{\frac{3}{2}^+}$ for $u_0$ with non zero negative part and local existence and singularity formation for $u_0\geq 0$ with some zero.

A higher dimensional version of \eqref{cht} with $\alpha=1$ was introduced by L. Caffarelli and J. L. V\'azquez in \cite{CV1} as a model of the dynamics of a gas in a porous media with a non local pressure. They studied the existence of weak solutions for initial data in $L^1$. See also \cite{CV2} for a proof of a regularizing effect. In  \cite{CSV}, L. Caffarelli, F. Soria and J.L. V\'azquez studied the case $0<\alpha<2$.

In \cite{CFP} J. A. Carrillo, L. Ferreira and J. Precioso  explore the gradient flow structure of \eqref{cht} with $\alpha=1$ (see also \cite{F}). For some further results concerning existence and singularity formation for related equations one can check \cite{LR}, by D. Li and J.L. Rodrigo, \cite{G}, by R. Granero-Belinch\'on, and \cite{BGL}, by H. Bae, R. Granero-Belinch\'on and O. Lazar, and references therein.

We shall present a blow up proof for equation \eqref{cht} with $\alpha=1$. This result was already proven in \cite{CC} (actually in \cite{CC} is proven a more general result) but we will include it here for the clarity of the exposition.

We take the initial data $u_0$ satisfying hypothesis H2 and H3. Thus the solution $u(x,t)$ also satisfies H2 and H3.

By using the identity
\begin{align*}
H(uHu)=\frac{1}{2}\p{Hu}^2-\frac{1}{2}u^2
\end{align*}
we get
\begin{align*}
\pa_t \Lambda u + \p{\Lambda u}^2+ Hu Hu_{xx}-u u_{xx}-\p{u_x}^2=0.
\end{align*}
By evaluating at $x=0$ yields
\begin{align*}
\pa_t \Lambda u(0,t)=-\p{\Lambda u (0,t)}^2 \quad \p{\Lambda u_0(0)<0},
\end{align*}
thus $\Lambda u(0,t)$ become infinity in finite time.

We shall emphasis that we have used hypothesis H2 and H3 to yields the inequality
\begin{align*}
-\int_{0}^\infty\frac{u(x)Hu(x)}{x^3}dx=\frac{1}{\pi}\p{\int_{0}^\infty \frac{u(x)}{x^2}dx}^2
\end{align*}
To show theorems \ref{thmline} and \ref{thmperiodic} we will prove a generalization of the previous inequality for $\alpha\neq 1$. We remark that this inequality is also valid in the periodic case.

\section{The case $1<\alpha<2$}\label{line}

In this section we will prove theorem \ref{thmline}.

We will take $\beta=\alpha-1$ and then we consider
\begin{align*}
\pa_t u +\left(u\Lambda^\beta H u \right)_x=0
\end{align*}
for $0<\beta<1$.

We  take $u_0\in C^{1+\beta^+}(\R)$ and
 we  assume that there exists a solution $u(x,t)\in C([0,\infty); C^{1+\beta^+}(\R))\cap C^1([0,\infty); C^{0^+}(\R))$. It can be checked that, then, $u(x,t)$ satisfies H1, H2 and H3 for all $t\in [0,\infty)$.

Multiplying equation \eqref{cht} by $x^{-(2+\beta)}$ and integrating from 0 to $\infty$ we have that
\begin{align}\label{uno}
\frac{d}{dt}\int_{0}^\infty \frac{u(x)}{x^{2+\beta}}dx = -\int_{0}^\infty\frac{\left(u(x) \Lambda^{\beta}Hu(x)\right)_x}{x^{2+\beta}}dx
=-(2+\beta)\int_{0}^\infty \frac{u(x)\Lambda^{\beta} H(u)}{x^{3+\beta}}dx.
\end{align}

Inspired by \cite{CCF},  where A. C\'ordoba, D. C\'ordoba and M.A. Fontelos showed the inequality
\begin{align}\label{CCFi}
-\int_{0}^\infty\frac{Hu(x) u_x(x)}{x^{1+\delta}}dx\geq K_\delta\p{\int_{\R}\frac{u(x)}{x^{2+\delta}}dx}^2,
\end{align}
for an even $C^1$-function compactly supported and $0<\delta<1$,  we will deal with the last integral in \eqref{uno} by using the Mellin Transform. Actually, in the following we will deal with a kind of endpoint of \eqref{CCFi}. Let us remind the definition and Parseval's identity of this transformation:
\begin{align*}
M[u](\lambda)=&\int_{0}^\infty x^{i\lambda-1}u(x)dx,\\
\int_{0}^\infty u(x)v(x)\frac{dx}{x}=&\frac{1}{2\pi}\int_{-\infty}^\infty M[u](\lambda)\overline{M[v](\lambda)} d\lambda.
\end{align*}

Thus
\begin{align*}
-(2+\beta)\int_{0}^\infty \frac{u(x)\Lambda^{\beta} H(u)}{x^{3+\beta}}dx=-\frac{(2+\beta)}{2\pi}\int_{-\infty}^\infty \overline{M[x^{-1+\ep}\Lambda^\beta Hu]}(\lambda)M[x^{-1-\beta-\ep}u](\lambda)d\lambda
\end{align*}
for any $0<\ep<\beta^+-\beta$.

We will use that $\Lambda^\beta H u =\Lambda^{\beta-1}\Lambda H u = -\Lambda^{\beta-1}u_x$. Thus

\begin{align*}
&M[x^{-1+\ep}\Lambda^\beta H u](\lambda)=-\int_{0}^\infty x^{i\lambda+\ep-2}\Lambda^{\beta-1}u_x(x)dx\\
&=-c_\beta\int_{0}^\infty x^{i\lambda+\ep-2}\int_{0}^\infty u_y(y)\left(\frac{1}{|x-y|^\beta}-\frac{1}{|x+y|^\beta}\right) dy dx\\
&=-c_\beta\int_{0}^\infty u_y(y)\int_{0}^\infty x^{i\lambda+\ep-2}\left(\frac{1}{|x-y|^\beta}-\frac{1}{|x+y|^\beta}\right) dx dy\\
&=-\int_{0}^\infty y^{i\lambda+\ep-\beta-1}u_y(y)c_\beta\int_{0}^\infty x^{i\lambda+\ep-2}\left(\frac{1}{|x-1|^\beta}-\frac{1}{|x+1|^\beta}\right) dx dy\\
&=m(\lambda,\ep,\beta)(i\lambda+\ep-1-\beta)M[x^{\ep-\beta-1}u](\lambda),
\end{align*}
where
\begin{align}\label{funcionm}
m(\lambda,\ep,\beta)=c_\beta\int_{0}^\infty x^{i\lambda+\ep-2}\p{\frac{1}{|x-1|^\beta}-\frac{1}{|x+1|^\beta}}dx.
\end{align}

Therefore
\begin{align*}
-(2+\beta)\int_{0}^\infty \frac{u(x)\Lambda^{\beta} H(u)}{x^{3+\beta}}dx=\frac{2+\beta}{2\pi}\int_{-\infty}^\infty B(\lambda,\ep,\beta)\overline{M[x^{\ep-1-\beta}u]}(\lambda)M[x^{-\ep-1-\beta}u](\lambda)d\lambda,
\end{align*}
where
\begin{align*}
B(\lambda,\ep,\beta)=\overline{m}(\lambda,\ep,\beta)(i\lambda-\ep+1+\beta).
\end{align*}

\begin{lemma}\label{mainlemma2} Let $u\in C_c^{1+\beta^+}(\R)$ such that $u(0)=u_x(0)=0$ and
\begin{align*}
U(\lambda,\ep,\beta)=\overline{M[x^{\ep-1-\beta}u]}(\lambda)M[x^{-\ep-1-\beta}u](\lambda),
\end{align*}
with $0<\beta<1$ and $0<\ep<\beta^+-\beta$ is a bounded function.

Then
\begin{align*}
\lim_{\ep\to 0^+} \int_{-\infty}^\infty B(\lambda,\ep,\beta)U(\lambda,\ep,\beta)=2\pi(1+\beta)\beta c_\beta U(0,0,\beta)+\int_{-\infty}^\infty B_0(\lambda,\beta) U(\lambda,0,\beta)d\lambda.
\end{align*}
where  $B_0(\lambda,\beta)\geq 0$.
\end{lemma}

\begin{cor}\label{maincoro} Let $u(x)$ be an odd function satisfying the assumptions of lemma \ref{mainlemma2}. Then the following estimate holds
\begin{align*}
-\int_{0}^\infty \frac{u(x)\Lambda^{\beta} H(u)}{x^{3+\beta}}dx\geq (1+\beta)\beta c_\beta\left(\int_{0}^\infty \frac{u(x)}{x^{2+\beta}}dx\right)^2
\end{align*}
\end{cor}

\begin{proof}

We will split the integral

 \begin{align}\label{I}
 I(\ep)=\int_{-\infty}^\infty B(\lambda,\ep,\beta)U(\lambda,\ep,\beta)d\lambda
 \end{align}
in two parts
\begin{align}\label{I1}
I_1(\ep)= \int_{-\infty}^\infty B(\lambda,\ep,\beta)U(\lambda,0,\beta)d\lambda,
\end{align}
and
\begin{align}\label{I2}
I_2(\ep)= \int_{-\infty}^\infty B(\lambda,\ep,\beta)\left(U(\lambda,\ep,\beta)-U(\lambda,0,\beta)\right)d\lambda,
\end{align}
thus $I(\ep)=I_1(\ep)+I_2(\ep)$. Here we recall that $U(\lambda,0,\beta)=|M[x^{-1-\beta} u]|^2(\lambda)$ is real and even function on $\lambda$.

We first pass to the limit $\ep\to 0^+$ in $I_1(\ep)$. We split $B(\lambda,\ep,\beta)$ into two parts
\begin{align*}
&B(\lambda,\ep,\beta)=(i\lambda-\ep+1+\beta)\int_{0}^\infty...\, dx\\& =(i\lambda-\ep+1+\beta)\left( \int_{0}^\frac{1}{2}...\, dx +\int_{\frac{1}{2}}^\infty...\, dx \right)\equiv  B_1(\lambda,\ep,\beta)+B_2(\lambda,\ep,\beta).\end{align*}
 With $B_1$ we proceed as follows. We first write
\begin{align*}
B_1(\lambda,\ep,\beta)=\frac{i\lambda-\ep+1+\beta}{-i\lambda+\ep-1}\int_0^\frac{1}{2} \pa_x x^{-i\lambda+\ep-1}P(x,\beta)dx,
\end{align*}
where $P(x,\beta)=c_\beta\left(\frac{1}{(1-x)^\beta}-\frac{1}{(1+x)^\beta}\right)$ is a smooth function on $x\in[0,\frac{1}{2}]$. Then, taking into account that $P(0,\beta)=0$, an integration by parts yields
\begin{align*}
B_1(\lambda,\ep,\beta)=-\frac{i\lambda-\ep+1+\beta}{-i\lambda+\ep-1}\int_0^\frac{1}{2}  x^{-i\lambda+\ep-1}P(x,\beta)_x dx+\frac{i\lambda-\ep+1+\beta}{-i\lambda+\ep-1}2^{i\lambda-\ep+1}P\left(\frac{1}{2},\beta\right).
\end{align*}
Integrating again by parts we have that

\begin{align*}
&B_1(\lambda,\ep,\beta)=\frac{i\lambda-\ep+1+\beta}{(-i\lambda+\ep-1)(-i\lambda+\ep)}\int_0^\frac{1}{2}  x^{-i\lambda+\ep}P(x,\beta)_{xx} dx-\frac{i\lambda-\ep+1+\beta}{(-i\lambda+\ep-1)(-i\lambda+\ep)}2^{i\lambda-\ep}P_{x}\left(\frac{1}{2},\beta\right)
\\&+\frac{-i\lambda-\ep+1+\beta}{-i\lambda+\ep-1}2^{i\lambda-\ep+1}P\left(\frac{1}{2},\beta\right)\equiv B_{11}(\lambda,\ep,\beta)+B_{12}(\lambda,\ep,\beta),
\end{align*}
with
\begin{align*}
&B_{11}(\lambda,\ep,\beta)=\frac{i\lambda-\ep+1+\beta}{(-i\lambda+\ep-1)(-i\lambda+\ep)}\left(\int_{0}^\frac{1}{2}x^{-i\lambda+\ep}P(x,\beta)_{xx} dx-2^{i\lambda-\ep}P_{x}\left(\frac{1}{2},\beta\right)\right),\\
&B_{12}(\lambda,\ep,\beta)=\frac{i\lambda-\ep+1+\beta}{-i\lambda+\ep-1}2^{i\lambda-\ep+1}P\left(\frac{1}{2},\beta\right).
\end{align*}

Next we consider the term in $I_1(\ep)$ involving $B_{11}$. Since $B_{11}$ is an odd function on $\lambda$ (and $U(\lambda,0,\beta)$ even) we have that
\begin{align*}
\int_{-\infty}^\infty B_{11}(\lambda,\ep,\beta)U(\lambda,0,\beta)d\lambda=\int_{-\infty}^\infty \text{Re}\left( B_{11}(\lambda,\ep,\beta)\right)U(\lambda,0,\beta)d\lambda.
\end{align*}

We notice that
\begin{align*}
&\frac{i\lambda-\ep+1+\beta}{(-i\lambda+\ep-1)(-i\lambda+\ep)}=-\frac{1}{-i\lambda+\ep}+\beta\frac{1}{(-i\lambda+\ep-1)(-i\lambda+\ep)}\\
&=-\frac{\ep+i\lambda}{\lambda^2+\ep^2}+\beta\frac{-\lambda^2+i\lambda(2\ep-1)+\ep(\ep-1)}{(\lambda^2+(\ep-1)^2)(\lambda^2+\ep^2)},\\
&=-\frac{\ep}{\lambda^2+\ep^2}+\beta\frac{-\lambda^2+\ep^2-\ep}{(\lambda^2+(\ep-1)^2)(\lambda^2+\ep^2)}+i\left(-\frac{\lambda}{\lambda^2+\ep^2}
+\frac{\lambda(2\ep-1)\beta}{(\lambda^2+(\ep-1)^2)(\lambda^2+\ep^2)}\right)\\
=&-\left(1+\frac{\beta}{\lambda^2+(1-\ep)^2}\right)\frac{\ep}{\lambda^2+\ep^2}+i\left(-1+\frac{(2\ep-1)\beta}{\lambda^2+(\ep-1)^2}\right)\frac{\lambda}{\lambda^2+\ep^2}
+\beta\frac{-\lambda^2+\ep^2}{(\lambda^2+(\ep-1)^2)(\lambda^2+\ep^2)}.
\end{align*}
Using this last expression we have that
\begin{align*}
&\text{Re}\left(B_{11}(\lambda,\ep,\beta)\right)\\&=-\left(1+\frac{\beta}{\lambda^2+(1-\ep)^2}\right)\frac{\ep}{\lambda^2+\ep^2}
\left(\int_{0}^\frac{1}{2}\cos\left(\lambda\log(x)\right)x^{-\ep}P_{xx}(x,\beta)dx-\cos(\lambda\log(2))2^{-\ep}P_x\left(\frac{1}{2},\beta\right)\right)\\
&+\left(-1+\frac{(2\ep-1)\beta}{\lambda^2+(\ep-1)^2}\right)\frac{\lambda}{\lambda^2+\ep^2}\left(\int_{0}^\frac{1}{2}
\sin\left(\lambda\log(x)\right)x^{-\ep}P_{xx}(x,\beta)dx+\sin(\lambda\log(2))2^{-\ep}P_x\left(\frac{1}{2},\beta\right)\right)\\
&+\beta\frac{-\lambda^2+\ep^2}{(\lambda^2+(\ep-1)^2)(\lambda^2+\ep^2)}\left(\int_{0}^\frac{1}{2}\cos\left(\lambda\log(x)\right)x^{-\ep}P_{xx}(x,\beta)dx-\cos(\lambda\log(2))2^{-\ep}P_x\left(\frac{1}{2},\beta\right)\right)
\end{align*}
We will split $\text{Re}B_{11}$ into two terms, denoting
\begin{align*}
&B^R_{111}(\lambda,\ep,\beta)\\&=-\left(1+\frac{\beta}{\lambda^2+(1-\ep)^2}\right)\frac{\ep}{\lambda^2+\ep^2}
\left(\int_{0}^\frac{1}{2}\cos\left(\lambda\log(x)\right)x^{-\ep}P_{xx}(x,\beta)dx-\cos(\lambda\log(2))2^{-\ep}P_x\left(\frac{1}{2},\beta\right)\right)\\
&\equiv -\left(1+\frac{\beta}{\lambda^2+(1-\ep)^2}\right)\frac{\ep}{\lambda^2+\ep^2}b_{111}(\lambda,\ep,\beta),\\
&B^R_{112}(\lambda,\ep,\beta)\\ &=+\left(-1+\frac{(2\ep-1)\beta}{\lambda^2+(\ep-1)^2}\right)\frac{\lambda}{\lambda^2+\ep^2}\left(\int_{0}^\frac{1}{2}
\sin\left(\lambda\log(x)\right)x^{-\ep}P_{xx}(x,\beta)dx+\sin(\lambda\log(2))2^{-\ep}P_x\left(\frac{1}{2},\beta\right)\right)\\
&+\beta\frac{-\lambda^2+\ep^2}{(\lambda^2+(\ep-1)^2)(\lambda^2+\ep^2)}\left(\int_{0}^\frac{1}{2}\cos\left(\lambda\log(x)\right)x^{-\ep}P_{xx}(x,\beta)dx-
\cos(\lambda\log(2))2^{-\ep}P_x\left(\frac{1}{2},\beta\right)\right)
\end{align*}
Next we will pass to the limit in the term inside of $I_1(\ep)$ involving $B^R_{111}(\lambda,\ep,\beta)$
\begin{align*}
&\lim_{\ep\to 0^+}\int_{-\infty}^\infty -\left(1+\frac{\beta}{\lambda^2+(1-\ep)^2}\right)\frac{\ep}{\lambda^2+\ep^2}b^R_{111}(\lambda,\ep,\beta)U(\lambda,0,\beta)d\lambda,\\
&=\lim_{\ep\to 0^+}\int_{-\infty}^\infty -\left(1+\frac{\beta}{\lambda^2+(1-\ep)^2}\right)\frac{\ep}{\lambda^2+\ep^2}b^R_{111}(\lambda,\ep,\beta)d\lambda\times U(0,0,\beta)\\
&+\lim_{\ep\to 0^+}\int_{-\infty}^\infty -\left(1+\frac{\beta}{\lambda^2+(1-\ep)^2}\right)\frac{\ep}{\lambda^2+\ep^2}b^R_{111}(\lambda,\ep,\beta)(U(\lambda,0,\beta)-U(0,0,\beta))d\lambda\\&
=\lim_{\ep\to 0^+}\int_{-\infty}^\infty -\left(1+\frac{\beta}{\ep^2\lambda^2+(1-\ep)^2}\right)\frac{1}{\lambda^2+1}b^R_{111}(\ep\lambda,\ep,\beta)d\lambda\times U(0,0,\beta)\\
&+\lim_{\ep\to 0^+}\int_{-\infty}^\infty -\left(1+\frac{\beta}{\ep^2\lambda^2+(1-\ep)^2}\right)\frac{1}{\lambda^2+1}b_{111}(\ep\lambda,\ep,\beta)(U(\ep\lambda,0,\beta)-U(0,0,\beta))d\lambda.
\end{align*}
By dominated convergence theorem (DCT) we get that
\begin{align*}
=-(1+\beta)\pi b^R_{111}(0,0,\beta)U(0,0,\beta).
\end{align*}
In addition
\begin{align*}
b^R_{111}(0,0,\beta)=\int_{0}^\frac{1}{2}P_{xx}(x,\beta)dx-P_x\left(\frac{1}{2},\beta\right)=-P_x(0,\beta)=-2\beta c_\beta.
\end{align*}
Therefore
\begin{align*}
&\lim_{\ep\to 0^+}\int_{-\infty}^\infty -\left(1+\frac{\beta}{\lambda^2+(1-\ep)^2}\right)\frac{\ep}{\lambda^2+\ep^2}b_{111}(\lambda,\ep,\beta)U(\lambda,0,\beta)d\lambda =2\pi(1+\beta)c_\beta\beta U(0,0,\beta)
\end{align*}
The rest of terms in $\text{Re}\left(B-B^R_{111}\right)(\lambda,\ep,\beta)$ are bounded by a constant uniformly in $\ep$ and $\lambda$ except by the term $B_{12}(\lambda,\ep,\beta)$ which is bounded by $C(1+|\lambda|)^{\beta^+}$ uniformly in $\ep$. In order to check these two facts are we notice that $B_{2}(\lambda,\ep,\beta)$ can be bounded by a constant in a trivial way and that the only factor in $B_{112}(\lambda,\ep,\beta)$ which can give some trouble is
\begin{align}\label{trouble}\frac{\lambda}{\lambda^2+\ep^2}\left(\int_{0}^\frac{1}{2}
\sin\left(\lambda\log(x)\right)x^{-\ep}P_{xx}(x,\beta)dx+\sin(\lambda\log(2))2^{-\ep}P_x\left(\frac{1}{2},\beta\right)\right).\end{align}
However,
\begin{align*}
\eqref{trouble}=\frac{\frac{\lambda}{\ep}}{1+\left(\frac{\lambda}{\ep}\right)^2}\left(-\int_{0}^\frac{1}{2}
\frac{\sin\left(\ep\frac{\lambda}{\ep}\log(x)\right)}{\ep}x^{-\ep}P_{xx}(x,\beta)dx-\frac{\sin(\ep\frac{\lambda}{\ep}\log(2))}{\ep}2^{-\ep}P_x\left(\frac{1}{2},\beta\right)\right).
\end{align*}
an then
\begin{align*}
\left|\eqref{trouble}\right|\leq \frac{\left(\frac{\lambda}{\ep}\right)^2}{1+\left(\frac{\lambda}{\ep}\right)^2}
\left(\int_{0}^\frac{1}{2} |\log(x)|x^{-\ep}P_{xx}(x,\beta)dx + \log(2)P_{x}\left(\frac{1}{2},\beta\right)\right)\leq C.
\end{align*}
For $B_2(\lambda,\ep,\beta)$ we have that
\begin{align*}
B_2(\lambda,\ep,\beta)=&(i\lambda-\ep+1+\beta)\int_{\frac{1}{2}}^\infty x^{-i\lambda+\ep-2}P(x,\beta)dx\\
&(i\lambda-\ep+1+\beta)\int_{\frac{1}{2}}^\frac{3}{2} x^{-i\lambda+\ep-2}P(x,\beta)dx+(i\lambda-\ep+1+\beta)\int_{\frac{3}{2}}^\infty x^{-i\lambda+\ep-2}P(x,\beta)\\
&\equiv B_{21}(\lambda,\ep,\beta)+B_{22}(\lambda,\ep,\beta).
\end{align*}
In $B_{22}(\lambda,\ep,\beta)$ we can integrate by parts to get
\begin{align*}
B_{22}(\lambda,\ep,\beta)=\frac{i\lambda-\ep+\beta+1}{-i\lambda+\ep-1}\int_{\frac{3}{2}}^\infty x^{-i\lambda+\ep-1}P_x(x,\beta)dx
-\frac{i\lambda-\ep+\beta+1}{-i\lambda+\ep-1}\left(\frac{3}{2}\right)^{-i\lambda+\ep-1}P\left(\frac{3}{2},\beta\right).
\end{align*}
from where we see that $|B_{22}(\lambda,\ep,\beta)|\leq C$ uniformly in $\ep$.

For $B_{21}(\lambda,\ep,\beta)$ we have that
\begin{align*}
B_{21}(\lambda,\ep,\beta)=&(i\lambda-\ep+1+\beta)\int_{\frac{1}{2}}^\frac{3}{2}x^{-i\lambda+\ep-2}|x-1|^{-\beta}dx\\
&- (i\lambda-\ep+1+\beta)\int_{\frac{1}{2}}^\frac{3}{2}x^{-i\lambda+\ep-2}(1+x)^{-\beta}dx\\
&\equiv B_{211}(\lambda,\ep,\beta)+B_{212}(\lambda,\ep,\beta).
\end{align*}
$B_{212}(\lambda,\ep,\beta)$ can be bounded in the same way than $B_{22}(\lambda,\ep,\beta)$ above.

With $B_{211}(\lambda,\ep,\beta)$ we proceed as follows. For $|\lambda|<1$ it is clear that $|B_{211}(\lambda,\ep,\beta)|<C$ uniformly in $\ep$. For $|\lambda|>1$
\begin{align*}
B_{211}(\lambda,\ep,\beta)=&\frac{(i\lambda-\ep+1+\beta)}{-i\lambda}\int_{\frac{1}{2}}^\frac{3}{2}\pa_x\left(x^{-i\lambda}-1\right)x^{\ep-2}|x-1|^{-\beta}dx\\
=&\frac{(i\lambda-\ep+1+\beta)}{-i\lambda}\int_{\frac{1}{2}}^\frac{3}{2}\left(x^{-i\lambda}-1\right)x^{\ep-2}\frac{\beta}{(x-1)|x-1|^\beta}dx\\
&+ \frac{(i\lambda-\ep+1+\beta)}{-i\lambda}\int_{\frac{1}{2}}^\frac{3}{2}\left(x^{-i\lambda}-1\right)(2-\ep)x^{\ep-3}|x-1|^{-\beta}dx\\
&+ \text{boundary terms at $x=\frac{1}{2}$ and $x=\frac{3}{2}$}.
\end{align*}
The boundary at $x=\frac{1}{2}$ and $x=\frac{3}{2}$ and the second term in the last equality of the last expression are bounded by a constant uniformly in $\ep$. Finally
\begin{align*}
\left|\int_{\frac{1}{2}}^\frac{3}{2}\left(x^{-i\lambda}-1\right)x^{\ep-2}\frac{\beta}{(x-1)|x-1|^\beta}dx\right|
\leq |\lambda|^{\beta^+}\int_{\frac{1}{2}}^\frac{3}{2}\frac{\left|x^{-i\lambda}-1\right|}{|\lambda|^{\beta^+}|x-1|^{\beta^+}}
x^{\ep-2}\frac{\beta}{|x-1|^{1+\beta-\beta^+}}dx
\end{align*}
where
$$\frac{\left|x^{-i\lambda}-1\right|}{|\lambda|^{\beta^+}|x-1|^{\beta^+}}\leq C$$
uniformly in $\lambda$ and $x\in [\frac{1}{2},\frac{3}{2}]$ for $\beta^+\leq 1$.
In addition
\begin{align*}
U(\lambda,0,\beta)&=\left|\int_{0}^\infty x^{i\lambda-\beta-2}u(x)dx\right|^2=
\frac{1}{\lambda^2+(1+\beta)^2}\left|\int_{0}^\infty x^{i\lambda-1-\beta}u_{x}(x)dx\right|^2\\
&\leq k\left(\beta,\beta^+||u||_{C^1},\text{support}\, u\right)\frac{1}{\lambda^2+(1+\beta)^2}
\end{align*}
Therefore, DCT applies to get
\begin{align*}
&\lim_{\ep\to 0^+}\int_{-\infty}^\infty\text{Re}\left(B-B_{111}\right)(\lambda,\ep,\beta)U(\lambda,0,\beta)d\lambda=\int_{-\infty}^\infty\lim_{\ep\to 0^+}\text{Re}\left(B-B_{111}\right)(\lambda,\ep,\beta)U(\lambda,0,\beta)d\lambda\\
&= \int_{-\infty}^\infty B_0(\lambda,\beta)U(\lambda,0,\beta)d\lambda,
\end{align*}
 where the real function $B_0(\lambda,\beta)$ is given by
\begin{align*}
B_0(\lambda,\beta)=&-\left(1+\frac{\beta}{1+\lambda^2}\right)\left(\int_{0}^\frac{1}{2}\frac{\sin(\lambda\log(x))}{\lambda}P_{xx}(x,\beta)dx+
\frac{\sin(\lambda\log(2))}{\lambda}P_x\left(\frac{1}{2},\beta\right)\right)\\
&-\beta\frac{1}{(\lambda^2+1)}\left(\int_{0}^\frac{1}{2}\cos\left(\lambda\log(x)\right)P_{xx}(x,\beta)dx-\cos(\lambda\log(2))
P_x\left(\frac{1}{2},\beta\right)\right)
\\
&+\text{Re}\left(\frac{i\lambda+1+\beta}{-i\lambda-1}2^{i\lambda+1}\right)P\left(\frac{1}{2},\beta\right)+\text{Re} \left((i\lambda+1+\beta)\int_{\frac{1}{2}}^\infty x^{i\lambda-2}P(x,\beta)dx\right).
\end{align*}

Next we prove that $\lim_{\ep\to 0^+} I_2(\ep)=0$. Since $\text{Re}(U(\lambda,\ep,\beta))$ is an even function on $\lambda$, proceeding as before we find that
\begin{align*}
\lim_{\ep\to 0^+}\int_{-\infty}^\infty B(\lambda,\ep,\beta) \text{Re}\big(U(\lambda,\ep,\beta)-U(\lambda,0,\beta)\big)d\lambda =0,
\end{align*}
then we just have to check
\begin{align*}
\lim_{\ep\to 0^+}\int_{-\infty}^\infty i B(\lambda,\ep,\beta)\text{Im}\left(U(\lambda,\ep,\beta)\right)d\lambda =0.
\end{align*}
Since $\text{Im} (U(\lambda,\ep,\beta))$ is an odd function in $\lambda$ and $\text{Re}(B(\lambda,\ep,\beta))$ is even we just need to  prove that
\begin{align*}
\lim_{\ep\to 0^+}\int_{-\infty}^\infty \text{Im}(B(\lambda,\ep,\beta))\text{Im}(U(\lambda,\ep,\beta))d\lambda =0.
\end{align*}
We will use that
\begin{align*}
\int_{0}^\infty x^{-i\lambda+\ep-2-\beta}u(x)dx=\frac{1}{i\lambda-\ep+1+\beta}\int_{0}^\infty x^{-i\lambda+\ep-1-\beta}u_x(x)dx.
\end{align*}
thus
\begin{align*}
&U(\lambda,\ep,\beta)=\frac{1}{i\lambda-\ep+1+\beta}\int_{0}^\infty x^{-i\lambda+\ep-1-\beta}u_x(x)dx\frac{1}{-i\lambda+\ep+1+\beta}\int_{0}^\infty x^{+i\lambda-\ep-1-\beta}u_x(x)dx\\
=&\frac{(-i\lambda-\ep+1+\beta)(i\lambda+\ep+1+\beta)}{(\lambda^2+(1+\beta-\ep)^2)(\lambda^2+(1+\beta+\ep)^2)}
\int_{0}^\infty \int_{0}^\infty\left(\frac{x}{y}\right)^{i\lambda-\ep}(xy)^{-1-\beta}u_x(x)u_x(y)dxdy\\
=&\frac{\lambda^2+(1+\beta)^2-\ep^2-2 i \ep\lambda}{(\lambda^2+(1+\beta-\ep)^2)(\lambda^2+(1+\beta+\ep)^2)}
\\ & \times\int_{0}^\infty \int_{0}^\infty\left(\cos\left(\lambda\log\left(\frac{x}{y}\right)\right)
+i\sin\left(\lambda\log\left(\frac{x}{y}\right)\right)\right)\left(\frac{x}{y}\right)^\ep(xy)^{-1-\beta}u_x(x)u_x(y)dxdy.
\end{align*}
Therefore
\begin{align}\label{IMU}
&\text{Im}(U(\lambda,\ep,\beta))\\
&=\frac{\lambda^2+(1+\beta)^2-\ep^2}{(\lambda^2+(1+\beta-\ep)^2)(\lambda^2+(1+\beta+\ep)^2)}
\int_{0}^\infty \int_{0}^\infty\left(\sin\left(\lambda\log\left(\frac{x}{y}\right)\right)\right)\left(\frac{x}{y}\right)^\ep(xy)^{-1-\beta}u_x(x)u_x(y)dxdy\nonumber\\
&-\frac{2 \ep\lambda}{(\lambda^2+(1+\beta-\ep)^2)(\lambda^2+(1+\beta+\ep)^2)}
\int_{0}^\infty \int_{0}^\infty\left(\cos\left(\lambda\log\left(\frac{x}{y}\right)\right)\right)\left(\frac{x}{y}\right)^\ep(xy)^{-1-\beta}u_x(x)u_x(y)dxdy\nonumber.
\end{align}
By making the same splitting $B=B_1+B_2$, $B_1=B_{11}+B_{12}$, $B_{11}=B_{111}+B_{112}$, we see that, by applying DCT,
\begin{align*}
&\lim_{\ep\to 0^+}\int_{-\infty}^\infty\text{Im}(B(\lambda,\ep,\beta)-B_{11}(\lambda,\ep,\beta))\text{Im}( U(\lambda,\ep,\lambda))d\lambda
\\&=\int_{-\infty}^\infty\lim_{\ep\to 0^+}\text{Im}(B(\lambda,\ep,\beta)-B_{111}(\lambda,\ep,\beta))\text{Im}( U(\lambda,\ep,\lambda))d\lambda =0,
\end{align*}
since $\lim_{\ep\to 0^+} \text{Im}(U(\lambda,\ep,\beta))=0$. And now the $\text{Im}(B_{11}(\lambda,\ep,\beta)$ contains harmless terms but those ones which contain  either the factor
\begin{align*}
\frac{\ep}{\lambda^2+\ep^2}\sin(\lambda\log(x) \qquad \text{or the factor} \qquad \frac{\lambda}{\lambda^2+\ep^2}\cos(\lambda\log(x).
\end{align*}
It easy to see that the term which contains the factor $\frac{\ep}{\lambda^2+\ep^2}\sin(\lambda\log(x)$ gives 0 in the limit $\ep\to 0^+$ because the $\sin(\lambda\log(x))$ in it. However to deal with the factor $\frac{\lambda}{\ep^2+\lambda^2}$ we need either the factor $\sin(\lambda\log(x/y))$ in the first term of \eqref{IMU} or the factor $\lambda$ in the second term of \eqref{IMU}. In any case we can show that DCT can be applied in order to get $\lim_{\ep\to 0^+} I_2(\ep)=0$.

\vspace{0.5cm}

Finally we will prove that $B_0(\lambda,\ep)\geq 0$. We notice that, for fixed $\lambda>0$,
\begin{align*}
B_0(\lambda,\beta)=\lim_{\ep\to 0^{+}}\text{Re}\left(c_\beta(i\lambda-\ep+1+\beta)\int_{0
}^\infty x^{-i\lambda+\ep-2}\left(\frac{1}{|x-1|^\beta}-\frac{1}{|1+x|^\beta}\right)dx\right).
\end{align*}
Thus  to show that $\text{Re}(B_0)(\lambda,\beta)\geq 0$ is enough to show that $\text{Re}(B)(\lambda,\ep,\beta)>0$  $\forall \ep>0$ and $\forall \lambda\neq 0$. In order to do it we first write $\overline{m}(\lambda,\ep,\beta)$ in the following way
\begin{align*}
\overline{m}(\lambda,\ep,\beta)=&c_{\beta}\int_{0}^\infty x^{i\lambda-2+\ep}\left(|x-1|^{-\beta}-|x+1|^{-\beta}\right)dx\\
&=c_{\beta}\int_{0}^1 x^{-i\lambda-2+\ep}\left(|x-1|^{-\beta}-|x+1|^{-\beta}\right)dx+c_{\beta}\int_{1}^\infty x^{-i\lambda-2+\ep}\left(|x-1|^{-\beta}-|x+1|^{-\beta}\right)dx\\
&=c_{\beta}\int_{0}^1 \left(x^{-i\lambda-2+\ep}+x^{i\lambda+\beta-\ep}\right)\left(|x-1|^{-\beta}-|x+1|^{-\beta}\right)dx\\
&=c_{\beta}\int_{0}^1 \cos(\lambda\log(x))\left(x^{-2+\ep}+x^{\beta-\ep}\right)\left(|x-1|^{-\beta}-|x+1|^{-\beta}\right)dx\\
&-c_{\beta}\int_{0}^1 \sin(\lambda\log(x))\left(x^{-2+\ep}-x^{\beta-\ep}\right)\left(|x-1|^{-\beta}-|x+1|^{-\beta}\right)dx
\end{align*}
where we did the change of variables $x'=1/x$. And then
\begin{align*}
&\text{Re}\p{(i\lambda+1+\beta-\ep)\overline{m}(\lambda, \ep,\beta)}=c_\beta \int_{0}^1 \cos(\lambda\log(x))(1+\beta-\ep)\left(x^{-2+\ep}+x^{\beta-\ep}\right)\left(|x-1|^{-\beta}-|x+1|^{-\beta}\right)dx\\
&+c_{\beta}\int_{0}^1 \lambda \sin(\lambda\log(x))\left(x^{-2+\ep}-x^{\beta-\ep}\right)\left(|x-1|^{-\beta}-|x+1|^{-\beta}\right)dx.
\end{align*}
Integrating by parts in the second integral in the previous expression we have that
\begin{align*}
\text{Re}\p{(i\lambda+1+\beta-\ep)\overline{m}(\lambda, \ep,\beta)}=c_\beta\int_{0}^1\cos(\lambda\log(x))F(x,\ep,\beta)dx
\end{align*}
with
\begin{align*}
F(x,\ep,\beta)=&(1+\beta-\ep)\left(x^{-2+\ep}+x^{\beta-\ep}\right)\left(|x-1|^{-\beta}-|x+1|^{-\beta}\right)\\&
+\pa_x\p{\left(x^{-1+\ep}-x^{1+\beta-\ep}\right)\left(|x-1|^{-\beta}-|x+1|^{-\beta}\right)}
\end{align*}

Integrating again we have that
\begin{align*}
\text{Re}\p{(i\lambda+1+\beta-\ep)\overline{m}(\lambda, \ep,\beta)}=&-\frac{c_\beta}{\lambda}\int_{0}^1\sin(\lambda\log(x))\pa_x(xF(x,\ep,\beta))dx.
\end{align*}
Let us call $G(x,\ep,\beta)=\pa_x\p{xF(x,\ep,\beta)}$. In order to be able to pass to the limit in $\ep$ we split
\begin{align*}
\int_{0}^1\sin(\lambda\log(x))G(x,\ep,\beta)dx=\int_{0}^\delta \sin(\lambda\log(x))G(x,\ep,\beta)dx+\int_{\delta}^1\sin(\lambda\log(x))G(x,\ep,\beta)dx.
\end{align*}
It is straightforward to pass to the limit in the second term of the previous equation. In the second term we will integrate by part. Since $G(x,\ep,b)\sim c_1 x^{-1+\ep}+c_2 x^{1+\ep}$ for $x\sim 0$ we have that
\begin{align*}
&\text{Im}\int_{0}^\delta x^{i\lambda}G(x,\ep,\beta)dx= \text{Im} \p{\frac{1}{i\lambda+\ep}\int_0^\delta \pa_x \p{x^{i\lambda+\ep}} x^{1-\ep}G(x,\ep,\beta)dx}\\
& =\text{Im}\p{ \frac{\delta^{i\lambda}}{i\lambda+\ep}}\delta  G(\delta,\beta,\ep)-\text{Im}\p{\int_{0}^\delta x^{i\lambda}x^{\ep}\pa_x\p{x^{1-\ep}G(x,\ep,\beta)}dx}
\end{align*}
and we pass to limit, $\ep\to 0$. After that we can pass to the limit $\delta\to 0$, and since the pointwise limit $G_0(x,\beta)=\lim_{\ep\to 0}G(x,\ep,\beta)$ (for $x>0$) satisfies $G(x,\beta)=O(x)$ for $x\sim 0$ we have that
\begin{align*}
B_0(\lambda,\beta)=-\frac{c_\beta}{\lambda}\int_{0}^1\sin\p{(\lambda\log(x))}G_0(x,\beta)dx.
\end{align*}
In addition and integration by parts yields,
\begin{align*}
B_0(\lambda,\beta)=\frac{c_\beta}{\lambda^2}\int_{0}^1(1-\cos(\lambda\log(x)))\pa_x(xG_0(x,\beta))dx
\end{align*}

Then, in order to prove that $B_0(\lambda,\beta)\geq$ we have to show that $G_0(x,\beta)+x\pa_x G_0(x,\beta)\geq 0$. Since $G(0,\beta)=0$  it is enough to prove that $\pa_x G_0(x,\beta)\geq 0$. Direct computations yields
\begin{align*}
\pa_x G_0(x,\beta)=\frac{\beta}{x^3}\p{f(x,\beta)-f(-x,\beta)},
\end{align*}
where
$$f(x,\beta)=(1-x)^{-3-\beta}\p{2+(3+\beta)x(-2+(2+\beta)x)-(1+\beta)(2+\beta)x|x|^{2+\beta}}.$$
In order to check that $\pa_x G_0(x,\beta)\geq 0$ we can write
\begin{align*}
f(x,\beta)-f(-x,\beta)=\int_{-1}^1\frac{d}{ds}f(sx,\beta)ds
\end{align*}
and check that
\begin{align*}
\pa_sf(sx,\beta)=(1+\beta)(2+\beta)(3+\beta)s^2x^3(1-sx)^{-4-\beta}\p{1-|s|^\beta x^\beta}\geq 0
\end{align*}
for $x>0$, $-1<s<1$.

\end{proof}

By applying corollary \ref{mainlemma2} we have that
\begin{align*}
\frac{d}{dt}\int_{0}^\infty \frac{u(x)}{x^{2+\beta}}dx\geq (2+\beta)(1+\beta)c_\beta\beta\left(\int_{0}^\infty \frac{u(x)}{x^{2+\beta}}dx\right)^2.
\end{align*}
Therefore $\int_{0}^\infty \frac{u(x)}{x^{2+\beta}}dx$ must blows up in finite time. However $\int_{0}^\infty \frac{u(x)}{x^{2+\beta}}dx\leq k_{\beta,\,\beta^+}||u||_{C^{1+\beta^+}}.$

\section{The case $0<\alpha<1$ in the periodic setting}\label{periodic}
We  take $u_0\in C^{1}(\R)$ and
 we  assume that there exist a solution $u(x,t)\in C([0,\infty); C^{1}(\R))\cap C^1([0,\infty); C(\R))$. It can be checked that, then, $u(x,t)$ satisfies H1, H2 and H3 for all $t\in [0,\infty)$. Also, and very important in this case, we will take $u_0$ satisfying H4. As it was proven in \cite{C} the solution also inherits this property, i.e., $$\pa_x u(x,t)\geq 0, \qquad  \text{for $x\in[0,\pi]$ and $t\geq 0$}.$$

In order to prove the existence of singularities we will use a similar strategy to one that in section \ref{line}. Let us first recall the following facts of the operator $\Lambda^{\alpha-1}$ in the periodic setting.

For a $2\pi-$periodic function $u(x)$ such that $\int_{\T}u(x)dx=0$, the operator $\Lambda^{\alpha-1}$ with $0<\alpha<1$ is defined through of the Fourier transform
\begin{align}\label{defifourier}
\widehat{\Lambda^{\alpha-1}u}(n)=|n|^{\alpha-1}\hat{u}(n)\quad 0< n\in \mathbb{N}.
\end{align}

 The operator $\Lambda^{\alpha-1}$ defined as \eqref{defifourier} admits the representation
\begin{align*}
\Lambda^{\alpha-1}u(x)=c_{\alpha}\int_{\R}\frac{u(y)}{|x-y|^{\alpha}}dy
\end{align*}
with $c_\alpha$ as in \eqref{Lambdadefi} and, for an even function $u$, $\Lambda^{\alpha-1}Hu$ can be written as
\begin{align*}
\overline{c}_\alpha\int_{0}^\infty u(y)\p{\frac{\sign(x-y)}{|x-y|^\alpha}+\frac{1}{|x-y|^\alpha}}dy, && \overline{c}_\alpha=\frac{\Gamma(\alpha)\sin\p{\frac{\pi\alpha}{2}}}{\pi}.
\end{align*}

We will look at the time evolution of $\int_{0}^\infty\frac{u(x)}{x^{1+\alpha}}dx$ which is given by
\begin{align*}
\pa_t \int_{0}^\infty\frac{u(x)}{x^{1+\alpha}}dx=&-(1+\alpha)\int_{0}^\infty\frac{u(x)\Lambda^{\alpha-1}Hu(x)}{x^{2+\alpha}}\\
& -(1+\alpha)\int_{0}^\infty x^{-\alpha-\ep} u(x)x^{-1}\Lambda^{\alpha-1+\ep}Hu(x)\frac{dx}{x}.
\end{align*}
Instead of $u(x)$ and $\Lambda^{\alpha-1}Hu$ do not decay, the functions $x^{-\alpha+\ep}u(x)$ and $x^{-1+\ep}\Lambda^{\alpha-1}u(x)$ have enough decay at  infinity (for $\ep$ small enough) to apply the Parseval Identity of the Mellin transform to get
\begin{align*}
\pa_t \int_{0}^\infty\frac{u(x)}{x^{1+\alpha}}dx=-\frac{1+\alpha}{2\pi}\int_{\R} \overline{M[x^{-\alpha-\ep}u(x)]}(\lambda)M[x^{-1+\ep}\Lambda^{\alpha-1}Hu(x)](\lambda)d\lambda,
\end{align*}
where
\begin{align*}
&M[x^{-\alpha-\ep}u(x)](\lambda)=\int_{0}^\infty x^{i\lambda-\alpha-1-\ep}u(x)dx\\
&M[x^{-1+\ep}\Lambda^{\alpha-1}Hu(x)](\lambda)=\int_{0}^\infty x^{i\lambda-2+\ep}\Lambda^{\alpha-1}Hu(x)dx
\end{align*}
Since $Hu(x)$ is an odd function of $x$, we can write
\begin{align*}
&M[x^{-1+\ep}\Lambda^{\alpha-1}Hu(x)](\lambda)=\overline{c}_\alpha\int_{0}^\infty u(y) \int_{0}^\infty x^{i\lambda-2+\ep}\p{\frac{\sign(x-y)}{|x-y|^\alpha}+\frac{1}{|x+y|^\alpha}} dx dy\\
&=\overline{c}_\alpha\int_{0}^\infty y^{i\lambda-\alpha+\ep-1} u(y) \int_{0}^\infty x^{i\lambda-2+\ep}\p{\frac{\sign(x-1)}{|x-1|^\alpha}+\frac{1}{|x+1|^\alpha}} dx dy\\
&=m_p(\lambda,\ep,\alpha)M[x^{-\alpha+\ep}u(x)](\lambda).
\end{align*}

Therefore we finally obtain that
\begin{align}\label{dos}
&-(1+\alpha)\int_{0}^\infty \frac{u(x)\Lambda^{\alpha-1} Hu(x) }{x^{2+\alpha}}dx\nonumber\\
&=-\frac{1+\alpha}{2\pi}\int_{-\infty}^\infty \overline{m}_p(\lambda,\ep,\alpha)
\overline{M[x^{\ep-\alpha}u]}(\lambda)M[x^{-\ep-\alpha}u](\lambda)d\lambda
\end{align}
We will denote
\begin{align*}
A(\lambda,\ep,\alpha)\equiv -\overline{m_p}(\lambda,\ep,\alpha)
\end{align*}
and
\begin{align*}
U(\lambda,\ep,\alpha)=\overline{M[x^{\ep-\alpha}u]}(\lambda)M[x^{-\ep-\alpha}u](\lambda).
\end{align*}

In order to prove the a lemma analogous to lemma \ref{mainlemma} in section \ref{line} we will need some preliminary results

\begin{lemma}\label{decaytotal} Let $|\lambda|>2$, $0<\alpha<1$. Then
\begin{enumerate}
\item $|\text{Im}\p{A(\lambda,\ep,\alpha)}|\leq C |\lambda|^{-1+\alpha}$.
\item $|\text{Re}\p{A(\lambda,\ep,\alpha)}|\leq C |\lambda|^{-2+\alpha}$.
\end{enumerate}
In both cases the constant $C$ does not depend either $\alpha$ and $\ep$.
\end{lemma}
\begin{proof}
Let $0<\delta<\frac{1}{2}$. At the end of the proof we will take $\delta=\lambda^{-1}$. Let us call $p(x,\alpha)=\frac{\sign(x-1)}{|x-1|^\alpha}+\frac{1}{|x-1|^\alpha}$.
To prove 1 we split
\begin{align}\label{I}
&I=\text{Im}\p{\frac{A(\lambda,\ep,\alpha)}{-\overline{c}_\alpha}}=\int_{0}^\infty \sin(\lambda \log(x))x^{-2+\ep}p(x,\alpha)dx\eqref{nonumber}\\
&=\int_{0}^{1-\delta}... dx + \int_{1-\delta}^{1+\delta}...dx+\int_{1+\delta}^\infty...dx\equiv I_{1}+I_2+I_3.\nonumber
\end{align}

Integrating by parts we have that
\begin{align*}
&I_1=-\frac{1}{\lambda}\int_{0}^{1-\delta} \pa_x \p{\cos(\lambda\log(x))} x^{-1+\ep}p(x,\alpha)dx=-\frac{1}{\lambda}\cos(\lambda\log(1-\delta))(1-\delta)^{-1+\ep}p((1-\delta),\alpha)\\
&+\frac{-1+\ep}{\lambda}\int_{0}^{1-\delta}\cos(\lambda\log(x))x^{-2+\ep}p(x,\alpha)dx+\frac{1}{\lambda}\int_{0}^{1-\delta}\cos(\lambda\log(x))x^{-1+\ep}\pa_x p(x,\alpha)dx\\
&\equiv I_{11}+I_{12}+I_{13}.
\end{align*}
\begin{align*}
&I_{12}=\frac{(-1+\ep)}{\lambda}\sin(\lambda\log(1-\delta))(1-\delta)^{-1+\ep}p((1-\delta),\alpha)
-\frac{(-1+\ep)^2}{\lambda^2}\int_{0}^{1-\delta} \sin(\lambda\log(x))x^{-2+\ep}p(x,\alpha)dx\\
&-\frac{(-1+\ep)}{\lambda^2}\int_{0}^{1-\delta}\sin(\lambda\log(x))x^{-1+\ep}\pa_x p (x,\alpha)dx= I_{121}-\frac{(-1+\ep)^2}{\lambda^2}I_{1}+I_{122}.
\end{align*}
Thus
\begin{align*}
\p{1+\frac{(-1+\ep)^2}{\lambda^2}}I_1=I_{11}+I_{13}+I_{121}+I_{122}.
\end{align*}
We now manipulate $I_{13}$ and $I_{122}$.
\begin{align*}
&I_{13}=\frac{1}{\lambda}\int_{0}^{1-\delta}\cos(\lambda\log(x))x^{-1+\ep}\pa_x p(x,\alpha)dx =\frac{1}{\lambda^2}\sin(\lambda\log(1-\delta))(1-\delta)^\ep\pa_xp((1-\delta),\alpha)\\
&-\frac{\ep}{\lambda^2}\int_{0}^{1-\delta}\sin(\lambda\log(x))x^{-1+\ep}\pa_x p(x,\alpha)dx-\frac{1}{\lambda^2}\int_{0}^{1-\delta}\sin(\lambda\log(x))x^\ep \pa^2_xp(x,\alpha)dx\\&=I_{131}+I_{132}+I_{133}.
\end{align*}
We integrate by parts in $I_{132}$ in such a way that
\begin{align*}
&I_{132}=-\frac{\ep}{\lambda^2}\int_{0}^{1-\delta}\sin(\lambda\log(x))x^{-1+\ep}\pa_x p(x,\alpha)dx
=\frac{\ep}{\lambda^3}\cos(\lambda\log(1-\delta))(1-\delta)^{\ep}\pa_x p((1-\delta),\alpha)dx\\
&+\frac{\ep^2}{\lambda^3}\int_{0}^{1-\delta}\cos(\lambda\log(x))x^{-1+\ep}\pa_xp(x,\alpha)dx
+\frac{\ep}{\lambda^3}\int_{0}^{1-\delta}\cos(\lambda\log(x))x^{\ep}\pa^2_xp(x,\alpha)dx\\
&=I_{1321}+\frac{\ep^2}{\lambda^2}I_{13}+I_{1322}
\end{align*}
Then
\begin{align*}
\p{1-\frac{\ep^2}{\lambda^2}}I_{13}=I_{131}+I_{1321}+I_{1322}+I_{133}.
\end{align*}
For $I_{122}$ we have that
\begin{align*}
&I_{122}=-\frac{(-1+\ep)}{\lambda^2}\int_{0}^{1-\delta}\sin(\lambda\log(x))x^{-1+\ep}\pa_x p (x,\alpha)dx
=\frac{(-1+\ep)}{\lambda^3}\cos(\lambda\log(1-\delta))(1-\delta)^\ep\pa_x p(1-\delta,\alpha)\\
&-\frac{(-1+\ep)\ep}{\lambda^3}\int_{0}^{1-\delta}\cos(\lambda\log(x))x^{-1+\ep}\pa_xp (x,\alpha)dx
-\frac{(-1+\ep)\ep}{\lambda^3}\int_{0}^{1-\delta}\cos(\lambda\log(x))x^{\ep}\pa^2_x p(x,\alpha)dx\\
&=I_{1221}+I_{1222}+I_{1223}.
\end{align*}
In $I_{1222}$ we integrate by parts to obtain that
\begin{align*}
&I_{1222}=-\frac{(-1+\ep)\ep}{\lambda^3}\int_{0}^{1-\delta}\cos(\lambda\log(x))x^{-1+\ep}\pa_xp (x,\alpha)dx=
-\frac{(-1+\ep)\ep}{\lambda^4}\sin(\lambda\log(1-\delta))(1-\delta)^\ep\pa_x p((1-\delta),\alpha)\\
&+\frac{(-1+\ep)\ep^2}{\lambda^4}\int_{0}^{1-\delta}\sin(\lambda\log(x))x^{-1+\ep} \pa_x p(x,\alpha)dx
+\frac{(-1+\ep)\ep}{\lambda^4}\int_{0}^{1-\delta}\sin(\lambda\log(x))x^\ep \pa^2_x p(x,\alpha)dx\\
=&I_{12221}-\frac{\ep^2}{\lambda^2}I_{122}+I_{12222}.
\end{align*}
Then
\begin{align*}
\p{1+\frac{\ep^2}{\lambda^2}}I_{122}=I_{1221}+I_{12221}+I_{12222}+I_{1223}.
\end{align*}

And now we can check that $|I_{1221}|\leq C|\lambda|^{-3}\delta^{-1-\alpha}$, $|I_{12221}|\leq C|\lambda|^{-4}\delta^{-1-\alpha}$, $|I_{12222}|\leq C |\lambda|^{-4}\delta^{-1-\alpha}$ and $|I_{1223}|\leq C |\lambda|^{-3}\delta^{-1-\alpha}$. Therefore $|I_{122}|\leq C |\lambda|^{-2+\alpha}$ if we take $\delta=|\lambda|^{-1}$.

In addition $|I_{131}|\leq |\lambda|^{-2}\delta^{-1-\alpha}$, $|I_{1321}|\leq C|\lambda|^{-3}\delta^{-1-\delta}$, $|I_{1322}|\leq C|\lambda|^{-3}\delta^{-1-\delta}$. Therefore $|I_{13}|\leq C|\lambda|^{-1+\alpha}$ if we take $\delta=|\lambda|^{-1}$.

Finally $|I_{11}|,\, |I_{121}|\leq C |\lambda|^{-1}\delta^{-\alpha}$. Thus $|I_1|\leq C |\lambda|^{-1-\alpha}$ with a constant $C$ that does not depend either $\ep$ or $\alpha$.

The term $I_{3}$ is easier to bound than $I_{1}$. But, if we want the constant $C$ independent of $\alpha$, we still have integrate by parts twice  to get
\begin{align*}
&I_{3}=\int_{1+\delta}^\infty \sin(\lambda\log(x)x^{-2+\ep}p(x,\alpha)=\frac{1}{\lambda}\cos(\lambda\log(1+\delta))(1+\delta)^{-1+\ep}p((1+\delta),\alpha)\\
&+\frac{(-1+\ep)}{\lambda}\int_{1+\delta}^\infty \cos(\lambda\log(x))x^{-2+\ep}p(x,\alpha)dx
+\frac{1}{\lambda}\int_{1+\delta}^\infty \cos(\lambda\log(x))x^{-1+\ep}\pa_xp(x,\alpha)dx\\
&=I_{31}+I_{32}+I_{33}.
\end{align*}
And
\begin{align*}
&I_{32}=\frac{(-1+\ep)}{\lambda}\int_{1+\delta}^\infty \cos(\lambda\log(x))x^{-2+\ep}p(x,\alpha)dx=
-\frac{-1+\ep}{\lambda^2}\sin(\lambda\log(1+\delta))(1+\delta)^{-2+\ep}p((1+\delta),\alpha)\\
&-\frac{(-1+\ep)^2}{\lambda^2}\int_{1+\delta}^\infty \sin(\lambda\log(x))x^{-2+\ep}p(x,\alpha)dx
-\frac{(-1+\ep)}{\lambda^2}\int_{1+\delta}^\infty \sin(\lambda\log(x))x^{-1+\ep}\pa_x p(x,\alpha)dx\\
&=I_{321}-\frac{(-1+\ep)^2}{\lambda^2}I_{33}+I_{322}.
\end{align*}
Thus
\begin{align*}
\p{1+\frac{(-1+\ep)^2}{\lambda^2}}I_{3}=I_{31}+I_{321}+I_{322}+I_{33}.
\end{align*}
And we have that $|I_{31}|\leq C |\lambda|^{-1}\delta^{-\alpha}$, $|I_{321}|\leq |\lambda|^{-2}\delta^{-\alpha}$, $|I_{322}|\leq |\lambda|^{-2}\delta^{-\alpha}$ and $|I_{33}|\leq |\lambda|^{-1}\delta^{-\alpha}$. Thus, by taking $\delta=|\lambda|^{-1}$ we have that $|I_3|\leq C|\lambda|^{-1+\alpha}$ with a constant $C$ that does not depend either on $\alpha$ or $\ep$.

The term $I_{2}$  can be bounded by $C \lambda \delta^{2-\alpha}+C\delta$ independently of $\alpha$ and $\ep$. This bound comes from the estimates
\begin{align*}
\int_{1-\delta}^{1+\delta} \sin\p{\lambda\log(x)}x^{-2+\ep}\frac{1}{|x+1|^\alpha}dx\leq C\delta
\end{align*}
and from
\begin{align*}
 &\int_{1-\delta}^{1+\delta}\sin\p{\lambda\log(x)}x^{-2+\ep}\frac{\sign(x-1)}{|x-1|^\alpha}dx\\
 &=\int_{0}^{\delta}\p{\sin\p{\lambda\log(1+x)}(1+x)^{-2+\ep}-\sin\p{\lambda\log(1-x)}(1-x)^{-2+\ep}}x^{-\alpha}dx\\
 &\leq C\lambda \delta^{2-\alpha}.
\end{align*}
Thus, by taking $\delta=|\lambda|^{-1}$ we have that $|I_2|\leq C|\lambda|^{-1+\alpha}$ with a constant $C$ that does not depend either on $\alpha$ or $\ep$.

We have already proven 1.

In order to prove 2 we notice that we have an extra integration by parts on $x$ for the real part. Indeed,

\begin{align*}
&J=\int_{0}^\infty \cos\p{\lambda\log(x)}x^{-2+\ep}p(x,\alpha)dx\\
&=\frac{(1-\ep)}{\lambda}\int_{0}^\infty \sin\p{\lambda\log(x)}x^{-2+\ep}p(x,\alpha)dx-\frac{1}{\lambda}\int_{0}^\infty \sin\p{\lambda\log(x)}x^{-1+\ep}\pa_x p(x,\alpha)dx\\
&=J_1+J_2
\end{align*}

We have that the integral in $J_1$ coincides with $I$ in \eqref{I}. Then we have already proved that $|J_1|\leq C |\lambda|^{-2+\alpha}$. We just have to deal with $J_2$.  We have that
\begin{align*}
&J_2=-\frac{1}{\lambda}\int_{0}^\infty \sin\p{\lambda\log(x)}x^{-1+\ep}\pa_x p(x,\alpha)dx=\frac{\alpha}{\lambda}\p{\int_{0}^{1-\delta}...\,dx+\int_{1-\delta}^{1+\delta}...\,dx+\int_{1+\delta}^\infty...\,dx}\\
&=J_{21}+J_{22}+J_{23}.
\end{align*}
In $J_{21}$ we can integrate by parts twice to get
\begin{align*}
J_{21}=&-\frac{(-1+\ep)}{\lambda^2}\int_{0}^{1-\delta}\cos\p{\lambda\log(x)}x^{-1+\ep}\pa_xp(x,\alpha)dx
-\frac{1}{\lambda^2}\int_{0}^{1-\delta}\cos\p{\lambda\log(x)}x^{\ep}\pa^2_x p(x,\alpha)dx\\
&+\frac{1}{\lambda^2}\cos(\lambda\log(1-\delta))(1-\delta)^\ep\pa_x p((1-\delta),\alpha)\\
=&-\frac{(-1+\ep)^2}{\lambda^2}J_{21}-\frac{(-1+\ep)}{\lambda^3}\int_{0}^{1-\delta}\sin(\lambda\log(x))x^\ep\pa^2_x p(x,\alpha)dx\\
&-\frac{(-1+\ep)}{\lambda^3}\sin(\lambda\log(1-\delta))(1-\delta)^\ep\pa_x p((1-\delta),\alpha)-\frac{1}{\lambda^2}\int_{0}^{1-\delta}\cos\p{\lambda\log(x)}x^{\ep}\pa^2_x p(x,\alpha)dx\\
&+\frac{1}{\lambda^2}\cos(\lambda\log(1-\delta))(1-\delta)^\ep\pa_x p((1-\delta),\alpha)
=-\frac{-1+\ep}{\lambda^2}J_{21}+J_{2122}+J_{2121}+J_{213}+J_{211}.
\end{align*}
Thus
\begin{align*}
\p{1+\frac{-1+\ep}{\lambda^2}}J_{21}= J_{211}+J_{2121}+J_{2122}+J_{213}
\end{align*}
which implies
\begin{align}\label{J21}
J_{21}=\p{1+\frac{1-\ep}{\lambda^2+(-1+\ep)}}\p{J_{211}+J_{2121}+J_{2122}+J_{213}}
\end{align}
We have that $|J_{211}|\leq C |\lambda|^{-2}\delta^{-1-\alpha} $, $|J_{213}|\leq C |\lambda|^{-2}\delta^{-1-\alpha}$, $|J_{2121}|\leq C|\lambda|^{-3}\delta^{-1-\delta}$ and $|J_{2122}|\leq C|\lambda|^{-3}\delta{-1-\delta}$. Therefore by taking $\delta=|\lambda|^{-1}|$ we see from  \eqref{J21} that $J_{21}=J_{211}+J_{213}+R$ with $|R|\leq C |\lambda|^{-2+\alpha}$.

Now we need to add $J_{21}$ and $J_{23}$ to find a cancellation between them. First of that, we integrate by parts on $J_{23}$
\begin{align*}
J_{23}=&\frac{1}{\lambda^2}\int_{1+\delta}^\infty \cos(\lambda\log(x))x^{-1+\ep}\pa_x p(x,\alpha)dx-\frac{1}{\lambda^2}\int_{1+\delta}^\infty \cos(\lambda\log(x))x^{\ep}\pa^2_x p(x,\alpha)dx\\
&-\frac{1}{\lambda^2}\cos(\lambda\log(1+\delta))\pa_x p((1-\delta),\alpha)=J_{232}+J_{233}+J_{231}.
\end{align*}
We have that $|J_{232}|\leq C |\lambda|^{-2}\delta^{-\alpha}$.

To bound $J_{211}+J_{213}$ we can use that $|\cos(\lambda\log(1+\delta))-\cos(\lambda\log(1-\delta))|\leq C\lambda^2\delta^3$ thus  $|J_{211}+J_{213}|\leq C \delta^{2-\alpha}$. For $J_{213}+J_{233}$ we can use the same fact. We just need to focus on the integrals
\begin{align*}
&-\int_{0}^{1-\delta}\cos(\lambda\log(x))x^{\ep}\frac{1}{|x-1|^{2+\alpha}}dx +\int_{1+\delta}^2\cos(\lambda\log(x))x^{\ep}\frac{1}{|x-1|^{2+\alpha}}dx\\
&=-\int_{-1}^{-\delta}\cos(\lambda\log(1+x))(1+x)^\ep\frac{1}{|x|^{2+\alpha}}dx +\int_{\delta}^1\cos(\lambda\log(1+x))(1+x)^\ep\frac{1}{|x|^{2+\alpha}}dx\\
&=\int_{\delta}^1\p{\cos(\lambda\log(1+x))(1+x)^\ep-\cos(\lambda\log(1-x))}(1-x)^\ep\frac{1}{|x|^{2+\delta}}dx
\end{align*}
Thus $|J_{213}+J_{233}|\leq C\delta^{2-\alpha}$.

It is remained to bound $J_{22}$ where
\begin{align*}
J_{22}=-\frac{1}{\lambda}\int_{1-\delta}^{1+\delta} \sin\p{\lambda\log(x)}x^{-1+\ep}\pa_x p(x,\alpha)dx
\end{align*}
We notice that it is enough to the study  the decay of the integral
\begin{align*}
H=&\frac{\alpha}{\lambda}\int_{1-\delta}^{1+\delta} \sin\p{\lambda\log(x)}x^{-1+\ep}\frac{1}{|x-1|^{1+\alpha}}dx\\
=&\frac{\alpha}{\lambda}\int_{-\delta}^{\delta} \sin\p{\lambda\log(1+x)}(1+x)^{-1+\ep}\frac{1}{|x|^{1+\alpha}}dx\\
=&\frac{\alpha}{\lambda}\int_{-\delta}^{\delta}\p{ \sin\p{\lambda\log(1+x)}-\sin(\lambda x)}(1+x)^{-1+\ep}\frac{1}{|x|^{1+\alpha}}dx\\
&+\frac{\alpha}{\lambda}\int_{-\delta}^{\delta} \sin\p{\lambda x}\p{(1+x)^{-1+\ep}-1}\frac{1}{|x|^{1+\alpha}}dx.\\
\end{align*}
Since $$|\sin\p{\lambda\log(1+x)}-\sin(\lambda x)|\leq C |\lambda|x^2$$
we can conclude that
\begin{align*}
|H|\leq C \delta^{2-\alpha}+C|\lambda|^{-1}\delta^{1-\alpha}.
\end{align*}

\end{proof}

\begin{lemma}\label{hurwitz} Let $u\in C^1(\T)$ an even function such that, $u(x)\geq 0$, $u_x(x)\geq 0$ for $x\in [0,\pi]$ and $u(0)=0$. Then
\begin{align*}
\left|\int_{0}^\infty x^{-i\lambda-1-\alpha}u(x)dx\right|\leq \frac{\alpha}{\sqrt{\lambda^2+\alpha^2}}\int_{0}^\infty x^{-1-\alpha}u(x)dx+ \frac{C+C_\epsilon\lambda^{1-\alpha +\epsilon}}{\sqrt{\lambda^2+\alpha^2}}||u||_{L^\infty(\T)}.
\end{align*}
\end{lemma}
\begin{proof} First we will integrate by parts to get
 \begin{align*}
 \int_{0}^\infty x^{-i\lambda-1-\alpha}u(x)dx=\frac{1}{-i\lambda+\alpha}\int_{0}^\infty x^{i\lambda-\alpha}u_x(x)dx.
 \end{align*}
We will use that $u_x(x)$ is a $2\pi-$periodic function to write
\begin{align*}
\int_{0}^\infty x^{-i\lambda-\alpha}u_x(x)dx=\sum_{n=0}^\infty \int_{2\pi n}^{2\pi(n+1)}x^{-i\lambda-\alpha}u_x(x)dx=\sum_{n=0}^\infty\int_{0}^{2\pi}\p{x+2\pi n}^{-i\lambda-\alpha}u_x(x)dx.
\end{align*}
and because $u(x)$ is even
\begin{align*}
&\int_{0}^{2\pi}x^{-i\lambda-\alpha}u_x(x)dx=\int_{0}^\pi (x+2\pi n)^{-i\lambda-\alpha}u_x(x)dx+\int_{\pi}^{2\pi} (x+2\pi n)^{-i\lambda-\alpha}u_x(x)dx\\
&=\int_{0}^\pi (x+2\pi n)^{-i\lambda-\alpha}u_x(x)dx+\int_{-\pi}^{2\pi}(x+2\pi(n+1)) ^{-i\lambda-\alpha}u_x(x)dx\\
&=\int_{0}^\pi \p{(x+2\pi n)^{-i\lambda-\alpha}-(-x+2\pi(n+1)) ^{-i\lambda-\alpha}}u_x(x)dx.
\end{align*}
Thus
\begin{align*}
\left|\int_{0}^\infty x^{-i\lambda-1-\alpha}u(x)dx\right|=\frac{1}{\sqrt{\lambda^2+\alpha^2}}\left|\int_{0}^\pi Z(i\lambda+\alpha,x)u_x(x)dx\right|,
\end{align*}
where
\begin{align*}
Z(i\lambda+\alpha,x)=&\sum_{n=0}^\infty (x+2\pi n)^{-i\lambda-\alpha}-(-x+2\pi(n+1))^{-i\lambda-\alpha}\\
&= x^{-i\lambda-\alpha}+\sum_{n=1}^\infty(x+2\pi n)^{-i\lambda-\alpha}-(-x+2\pi n)^{-i\lambda-\alpha}\\
&=x^{-i\lambda-\alpha}+(2\pi)^{-i\lambda-\alpha}\sum_{n=1}^\infty \p{n+\frac{x}{2\pi}}^{-i\lambda-\alpha}-\p{n-\frac{x}{2\pi}}^{-i\lambda-\alpha}\\
\end{align*}
Let us write $Z^*(i\lambda+\alpha)=\sum_{n=1}^\infty \p{n+\frac{x}{2\pi}}^{-i\lambda-\alpha}-\p{n-\frac{x}{2\pi}}^{-i\lambda-\alpha}$. We can split $Z^*=Z^*_1+Z^*_2$ with
\begin{align*}
Z^*_1=\sum_{n=1}^\infty \p{\p{n+\frac{x}{2\pi}}^{-\alpha}-\p{n-\frac{x}{2\pi}}^{-\alpha}}\p{n+\frac{x}{2\pi}}^{-i\lambda}\\
Z^*_2=\sum_{n=1}^\infty \p{\p{n+\frac{x}{2\pi}}^{-i\lambda}-\p{n-\frac{x}{2\pi}}^{-i\lambda}}\p{n-\frac{x}{2\pi}}^{-\alpha}\\
\end{align*}
Since $0\leq \frac{x}{2\pi}\leq \frac{1}{2}$ we have that $|Z^*_1|\leq C$ where $C$ does not depend on neither $\alpha$ nor $x$. For $Z^*_2$ we have that
\begin{align*}
Z_2^*= & \sum_{n=1}^\infty \p{n-\frac{x}{2\pi}}^{-\alpha}\p{\cos\p{\lambda\log\p{n+\frac{x}{2\pi}}}-\cos\p{\lambda\log\p{n-\frac{x}{2\pi}}}}\\
& +i \sum_{n=1}^\infty \p{n-\frac{x}{2\pi}}^{-\alpha}\p{\sin\p{\lambda\log\p{n+\frac{x}{2\pi}}}-\sin\p{\lambda\log\p{n-\frac{x}{2\pi}}}}
\end{align*}
We will bound the real part of $Z_2^*$. The estimation of the imaginary part follows same steps. Using elementary trigonometric formulas we can write
\begin{align*}
\text{Re}\,Z_2^* =2\sum_{n=1}^\infty \p{n-\frac{x}{2\pi}}^{-\alpha}\cos\p{\frac{\lambda}{2}\log\p{\p{n+\frac{x}{2\pi}}\p{n-\frac{n}{2\pi}}}}
\sin\p{\frac{\lambda}{2}\log\p{\frac{n+\frac{x}{2\pi}}{n-\frac{x}{2\pi}}}}.
\end{align*}
We notice that $$\frac{n+\frac{x}{2\pi}}{n-\frac{x}{2\pi}}=1+\frac{\frac{x}{\pi n}}{1-\frac{x}{2\pi n}}$$ and that for all $0<\beta<1$ , $|\sin(x)\leq |x|^\beta$, thus
 \begin{align*}
 \left|\sin\p{\frac{\lambda}{2}\log\p{\frac{n+\frac{x}{2\pi}}{n-\frac{x}{2\pi}}}}\right|\leq \left|\frac{\lambda}{2}\log\p{\frac{n+\frac{x}{2\pi}}{n-\frac{x}{2\pi}}}\right|^\beta\leq \left|\frac{\frac{\lambda x}{2\pi n}}{1-\frac{x}{2\pi n}}\right|^\beta
\leq C \left|\frac{\lambda}{n}\right|^{\beta}.
 \end{align*}
Taking $\beta=1-\alpha+\epsilon$ we have that
\begin{align*}
|\text{Re}\, Z_2^*|\leq C|\lambda|^{1-\alpha+\epsilon}\sum_{n=1}^\infty \left|n-\frac{x}{2\pi}\right|^{-\alpha}n^{-1+\alpha-\ep}\leq C_{\epsilon}|\lambda|^{1-\alpha+\epsilon}.
\end{align*}
Proceeding in a similar way with $\text{Im} Z_2^*$ we have that $|Z^*(i\lambda+\alpha)|\leq C +C_{\epsilon}|\lambda|^{1-\alpha+\epsilon}.$

 This fact allows us to prove that
\begin{align*}
\left|\int_{0}^\infty x^{-i\lambda-1-\alpha} u(x)dx\right|\leq \frac{1}{\sqrt{\lambda^2+\alpha^2}}\int_{0}^\pi x^{-\alpha} \left|u_x(x)\right|dx+\frac{C+C_\epsilon |\lambda|^{1-\alpha+\epsilon}}{\sqrt{\lambda^2+\alpha^2}}\int_{0}^\pi  \left|u_x(x)\right|dx
\end{align*}
But since $u_x(x)\geq 0 $ for $x\in [0,\pi]$, we can remove the absolute value inside the integral to get
\begin{align*}
\left|\int_{0}^\infty x^{-i\lambda-1-\alpha} u(x)dx\right|\leq \frac{1}{\sqrt{\lambda^2+\alpha^2}}\int_{0}^\pi x^{-\alpha} u_x(x)dx+\frac{C+C_{\epsilon}|\lambda|^{1-\alpha+\epsilon}}{\sqrt{\lambda^2+\alpha^2}}||u||_{L^\infty}.
\end{align*}
To finish the proof we notice that
$$Z(\alpha,x)=x^{-\alpha}+\sum_{n=1}^\infty \p{(x+2\pi n)^{-\alpha}-(-x+2\pi n)^{-\alpha}}$$
is a positive function that satisfies
$$x^{-\alpha}\leq Z(\alpha,x)+C.$$
Therefore
\begin{align*}
&\left|\int_{0}^\infty x^{-i\lambda-1-\alpha} u(x)dx\right|\leq \frac{1}{\sqrt{\lambda^2+\alpha^2}}\int_{0}^\pi Z(\alpha,x) u_x(x)dx+\frac{C+C_{\epsilon}|\lambda|^{1-\alpha+\epsilon}}{\sqrt{\lambda^2+\alpha^2}}||u||_{L^\infty}\\
&=\frac{1}{\sqrt{\lambda^2+\alpha^2}}\int_{0}^\infty x^{-\alpha} u_x(x)dx+\frac{C+C_{\epsilon}|\lambda|^{1-\alpha+\epsilon}}{\sqrt{\lambda^2+\alpha^2}}||u||_{L^\infty}.
\end{align*}
Then we can achieve the conclusion of the lemma just integrating by parts.

\end{proof}

Now we can prove the main lemma of this section
\begin{lemma} \label{mainlemma} Let $u\in C^1(\T)$ satisfying, H2, H3 and H4. Then
\begin{align*}
\lim_{\ep\to 0^+} \int_{-\infty}^\infty A(\lambda,\ep,\alpha)U(\lambda,\ep,\alpha)d\lambda= 2\pi\alpha\overline{c}_\alpha U(0,0,\alpha)+\int_{-\infty}^\infty A_0(\lambda,\alpha)U(\lambda,0,\alpha)d\lambda.
\end{align*}
where
\begin{enumerate}
\item [P1.] $$A_0(\lambda,\alpha)=-\overline{c}_\alpha\lim_{\ep\to 0^{+}}\text{Re}\int_{0}^\infty x^{-i\lambda-2+\ep}\p{\frac{\sign(x-1)}{|x-1|^\alpha}+\frac{1}{|x+1|^\alpha}}dx\quad \text{for $|\lambda|>0$}.$$
\item [P2.] $A_0(\lambda,\alpha)\leq 0$ is a bounded function.
\end{enumerate}
\end{lemma}

\begin{proof}
The proof of this lemma is similar to that one for lemma \ref{mainlemma} in section \ref{line}. The main difference is that in \ref{mainlemma} we could use that the integral $\int_{0}^\infty x^{i\lambda-2-\beta}u(x)$ decays as $|\lambda|^{-1}$ and then $U(\lambda,\ep,\beta)$ decays as $|\lambda|^{-2}$. Since this is not the case now we have to use the decays in lemmas \ref{decaytotal} and \ref{hurwitz} to get that
\begin{align}\label{cae}|A(\lambda,\ep,\alpha)U(\lambda,\ep,\alpha)|\leq C_{\epsilon} |\lambda|^{-1-\alpha+\epsilon}\end{align}
with $C_{\epsilon}$ independent of $\alpha$ and $\ep$, $\epsilon$ arbitrarily small and $|\lambda|>2$. Thus

\begin{align*}
&\int_{-\infty}^\infty A(\lambda,\ep,\alpha)U(\lambda,\ep,\alpha)d\lambda=\text{Re}\p{\int_{-\infty}^\infty A(\lambda,\ep,\alpha)U(\lambda,\ep,\alpha)d\lambda}\\
=&\int_{-\infty}^\infty \text{Re}\p{A(\lambda,\ep,\alpha)}\text{Re}\p{U(\lambda,\ep,\alpha)}d\lambda -\int_{-\infty}^\infty \text{Im}\p{A(\lambda,\ep,\alpha)}\text{Im}\p{U(\lambda,\ep,\alpha)}d\lambda.
\end{align*}
Now we can split the integrals in $\int_{|\lambda|>2}...+\int_{|\lambda|<2}d\lambda$. Because of \eqref{cae} we have that
\begin{align*}
\lim_{\ep\to 0^+}\int_{|\lambda|>2}\text{Re}\p{A(\lambda,\ep,\alpha)}\text{Re}\p{U(\lambda,\ep,\alpha)}d\lambda=
\int_{|\lambda|>2}A_0(\lambda,\alpha)U(\lambda,0,\alpha)d\lambda,
\end{align*}
and
\begin{align*}
\lim_{\ep\to 0^+}\int_{|\lambda|>2}\text{Im}\p{A(\lambda,\ep,\alpha)}\text{Im}\p{U(\lambda,\ep,\alpha)}d\lambda=0.
\end{align*}
For the integral on the region $|\lambda|<2$ we can perform similar computations to that one in lemma \ref{mainlemma} to achieve the conclusion of the lemma. In order to make shorter the proof we will skip some details and we will use $a^\ep(\lambda)\sim b^\ep(\lambda)$ if $$\lim_{\ep\to 0^+}\int_{|\lambda|<2}(a^{\ep}(\lambda)-b^\ep(\lambda))f(\lambda)d\lambda =\int_{|\lambda|<2}c(\lambda)f(\lambda)$$ for any smooth function $f$ and with $c(\lambda)$ a bounded function.

We focus on the integral
\begin{align*}
\int_{0}^\infty x^{i\lambda-2+\ep}p(x,\alpha)dx\sim \int_{0}^\frac{1}{2} x^{i\lambda-2+\ep}p(x,\alpha)dx,
\end{align*}
with $p(x,\alpha)=\sign(x-1)|x-1|^{-\alpha}+|x+1|^\alpha$.

And integration by parts yields
\begin{align*}
&\int_{0}^\frac{1}{2}x^{i\lambda-2+\ep}p(x,\alpha)dx=\frac{1}{i\lambda-1+\ep} 2^{-i\lambda+1-\ep}p(2^{-1},\alpha)-\frac{1}{i\lambda-1+\ep}
\int_{0}^\frac{1}{2}x^{i\lambda-1+\ep}\pa_x p(x,\alpha) dx\\
&\sim -\frac{1}{i\lambda-1}
\int_{0}^\frac{1}{2}x^{i\lambda-1+\ep}\pa_x p(x,\alpha) dx=-\frac{1}{i\lambda+\ep}2^{i\lambda+\ep}\pa_x p(2^{-1},\alpha)+\frac{1}{i\lambda-1}
\frac{1}{i\lambda+\ep}\int_{0}^\frac{1}{2}x^{i\lambda+\ep}\pa^2_x p(x,\alpha)dx\\
&\sim -\frac{1}{i\lambda-1}\frac{1}{i\lambda+\ep}2^{i\lambda+\ep}\pa_x p(2^{-1},\alpha)+\frac{1}{i\lambda-1}
\frac{1}{i\lambda+\ep}\int_{0}^\frac{1}{2}x^{i\lambda+\ep}\pa^2_x p(x,\alpha)dx\\
&\sim -\frac{1}{i\lambda-1}\frac{1}{i\lambda+\ep}2^{i\lambda}\pa_x p(2^{-1},\alpha)+\frac{1}{i\lambda-1}
\frac{1}{i\lambda+\ep}\int_{0}^\frac{1}{2}x^{i\lambda}\pa^2_x p(x,\alpha)dx
\end{align*}
Then
\begin{align*}
\text{Re}\p{\int_{0}^\infty x^{i\lambda-2+\ep}p(x,\alpha)dx}\sim \frac{\ep}{\ep^2+\lambda^2}\frac{1}{1+\lambda^2}\p{\cos(\lambda\log(2))\pa_x  p(2^{-1},\alpha)-
\int_{0}^\frac{1}{2}\cos(\lambda\log(x))\pa^2_x\pa_x^2 p(x,\alpha)dx}.
\end{align*}
From this last expression we can conclude that
\begin{align*}
\lim_{\ep\to 0^+}\int_{|\lambda|<2} A(\lambda,\ep,\alpha)U(\lambda,\ep,\alpha)d\lambda =-\pi\overline{c}_\alpha\pa_x p(0,\alpha)U(0,0,\alpha)+\int_{|\lambda|<2}A_0(\lambda,\alpha)U(\lambda,0,\alpha)d\lambda,
\end{align*}
where $A_0(\alpha,\lambda)$ is a bounded function and $\pa_x p(0,\alpha)=-2\alpha$.

In addition we have that
\begin{align*}
\text{Im}\p{\int_{0}^\infty x^{i\lambda-2+\ep}p(x,\alpha)dx}\sim \frac{\lambda}{\ep^2+\lambda^2}g(\lambda)
\end{align*}
where $g(\lambda)$ is a bounded function. To pass to the limit in
\begin{align*}
\int_{|\lambda|<2} \text{Im}(A(\lambda,\ep,\alpha))\text{Im} (U(\lambda,\ep,\alpha)d\lambda
\end{align*}
we take advantage of the fact that $Im(U(\lambda,\ep,\alpha)d\lambda\leq C|\lambda|$. We finally have that
\begin{align*}
\lim_{\ep\to 0^+}\int_{|\lambda|<2} \text{Im}(A(\lambda,\ep,\alpha))\text{Im} (U(\lambda,\ep,\alpha)d\lambda=0.
\end{align*}

 Finally we prove that negativity of $A_0(\lambda,\alpha)$. This proof follows similar steps that the proof of the positivity of $B_0(\lambda,\beta)$ in lemma \ref{mainlemma}. We have that, for $\lambda>0$,
 \begin{align*}
I= \int_{0}^\infty \cos(\lambda\log(x))p(x,\alpha)dx =\int_{0}^1\cos(\lambda\log(x)) F(x,\ep,\alpha)dx,
 \end{align*}
 where $$F(x,\ep,\alpha)=x^{-2+\ep}\p{-(1-x)^{-\alpha}+(1+x)^{-\alpha}}+x^{\alpha-\ep}\p{(1-x)^{-\alpha}+(1+x)^{-\alpha}}.$$
We integrate by parts to have that
\begin{align*}
I=-\frac{1}{\lambda}\int_{0}^1\sin(\lambda\log(x))\pa_x\p{xF(x,\ep,\alpha)}dx.
\end{align*}
For $\lambda>0$ we pass to the limit as we did for $B_0(\lambda,\beta)$ in lemma \ref{mainlemma}. This yields
\begin{align*}
\lim_{\ep\to 0}I=-\frac{1}{\lambda}\int_{0}^1\sin(\lambda\log(x))G(x,\alpha)dx,
\end{align*}
where
$G(x,\alpha)=\pa_x(xF(x,0,\alpha))$  is positive and its derivative $\pa_x G(x,\alpha)$ is also positive. Then an integration by part shows that $\lim_{\ep\to 0}I\geq0$. Since $A_0(\lambda,\alpha)=-\overline{c}_\alpha\lim_{\ep\to 0^+} I$ we have achieved the last conclusion of the lemma.
\end{proof}

\begin{rem}The main difference between the ranges $0<\alpha<1$ and $1<\alpha<2$ is that in lemma \ref{mainlemma2} the function $B_0(\lambda,\alpha)$ is positive and in lemma \ref{mainlemma} the function $A_0(\lambda,\alpha)$ is negative.
\end{rem}

Applying lemma \ref{mainlemma} to \eqref{dos} yields
\begin{align}\label{esta}
&-(1+\alpha)\int_{0}^\infty \frac{u(x)\Lambda^{\alpha-1} Hu(x) }{x^{2+\alpha}}dx\nonumber\\
&=\frac{1+\alpha}{2\pi}\left(C(\alpha)U(0,0,\alpha)+\int_{-\infty}^\infty A_0(\lambda,\alpha)U(\lambda,0,\alpha)d\lambda\right),
\end{align}
where
\begin{align*}
U(\lambda,0,\alpha)=\left|\int_{0}^\infty x^{i\lambda-1-\alpha}u(x)dx\right|^2.
\end{align*}

Unlike the function $B_0(\lambda,\alpha)$ which is positive, the function $A_0(\lambda,\alpha)$ is negative an then we have to work further to be able to conclude the blow up of solutions.

We will need the following lemma
\begin{lemma} \label{decay} Let $0<\alpha<1$ and $|\lambda|> 2$. Then
\begin{align*}
\left|A_0(\lambda,\alpha)\right|\leq \frac{C}{|\lambda|^{2-\alpha}}
\end{align*}
where $C$ is a universal constant which not depend on $\alpha$.
\end{lemma}
\begin{proof}
This lemma is a consequence of lemma \ref{decaytotal}
\end{proof}

From \eqref{esta} we find that
\begin{align*}
\frac{d}{dt}\int_{0}^\infty \frac{u(x,t)}{x^{1+\alpha}}dx=&\frac{(1+\alpha)}{2\pi}2\pi\alpha\overline{c}_\alpha\p{\int_{0}^\infty \frac{u(x,t)}{x^{1+\alpha}}dx}^2\\
&+\frac{(1+\alpha)}{2\pi}\int_{-\infty}^\infty A_0(\lambda,\alpha)\left|\int_{0}^\infty x^{-i\lambda-1-\alpha} u(x)dx\right|^2 d\lambda.
\end{align*}
Applying lemma \ref{hurwitz} yields,

\begin{align*}
&\int_{-\infty}^\infty A_0(\lambda,\alpha)\left|\int_{0}^\infty x^{-i\lambda-1-\alpha} u(x)dx\right|^2 d\lambda\geq
-\p{\int_{-\infty}^\infty |A_0(\lambda,\alpha)|\frac{\alpha^2}{\lambda^2+\alpha^2}d\lambda} \p{\int_{0}^\infty \frac{u(x)}{x^{1+\alpha}}dx}^2\\&-||u||_{L^\infty}\p{\int_{\R} |A_0(\lambda,\alpha)|\frac{\alpha (C+C_\epsilon|\lambda|^{1-\alpha+\epsilon})}{\lambda^2+\alpha^2}d\lambda}\int_{0}^\infty \frac{u(x)}{x^{1+\alpha}}dx \\
&-||u||^2_{L^\infty}\p{\int_{-\infty}^\infty|A_0(\lambda,\alpha)|\frac{\p{C+C_\epsilon|\lambda|^{1-\alpha+\ep}}^2}{\lambda^2+\alpha^2}d\lambda}.
\end{align*}

By applying the maximum principle for solutions of \eqref{cht} we conclude that

\begin{align*}
\frac{d}{dt}\int_{0}^\infty \frac{u(x,t)}{x^{1+\alpha}}dx= C_1(\alpha)\p{\int_{0}^\infty \frac{u(x)}{x^{1+\alpha}}dx}^2+||u_0||_{L^\infty}C_2^\epsilon(\alpha)+||u_0||_{L^\infty}^2C_{3}^\epsilon(\alpha).
\end{align*}
where
\begin{align*}
C_1(\alpha)=&\frac{1+\alpha}{2\pi}\p{2\pi \overline{c}_\alpha+\int_{-\infty}^\infty A_0(\alpha,\lambda)\frac{\alpha^2}{\lambda^2+\alpha^2}d\lambda}\\
C_2^\epsilon(\alpha)=&\int_{-\infty}^\infty|A_0(\lambda,\alpha)|\frac{\alpha (C+C_\epsilon|\lambda|^{1-\alpha+\epsilon})}{\lambda^2+\alpha^2}\\
C_3^\epsilon(\alpha)=&\int_{-\infty}^\infty|A_0(\lambda,\alpha)|\frac{(C+C_\epsilon|\lambda|^{1-\alpha+\epsilon})^2}{\lambda^2+\alpha^2}
\end{align*}
where we took into account that $A_0(\alpha,\lambda)\leq 0$.

We firstly analyze $C_2^\epsilon(\alpha)$ and $C_3^\epsilon(\alpha)$. We can split
\begin{align*}
C_2^\epsilon(\alpha)=\int_{|\lambda|<2} ... d\lambda +\int_{|\lambda|>2} ... d\lambda.
\end{align*}
And we can bound, for  $\epsilon<1$,
\begin{align*}
\int_{|\lambda|<1}|A_0(\lambda,\alpha)|\frac{\alpha (C+C_\epsilon|\lambda|^{1-\alpha+\epsilon})}{\lambda^2+\alpha^2}\leq C_\epsilon\int_{|\lambda|<2}\frac{\alpha}{\alpha^2+\lambda^2}d\lambda\leq C_{\epsilon}.
\end{align*}
In addition, by lemma \ref{decay}, for $\epsilon < 1$,
\begin{align*}
\int_{|\lambda|>2}|A_0(\lambda,\alpha)|\frac{\alpha (C+C_\epsilon|\lambda|^{1-\alpha+\epsilon})}{\lambda^2+\alpha^2}\leq C
\int_{|\lambda|>2}|\lambda|^{-2+\alpha}\frac{C+C_\epsilon |\lambda|^{1-\alpha+\epsilon}}{\lambda^2+\alpha^2}d\lambda\leq C_{\epsilon}.
\end{align*}
Therefore $C^\epsilon_2(\alpha)$ can be choosen less or equal than a constant $C_2^\epsilon$ uniformly in $\alpha$ and $\epsilon<1$.

For $C^\epsilon_3(\alpha)$ and $\ep<1$ we have that
\begin{align*}
\int_{|\lambda|<2}|A_0(\alpha,\lambda)|\frac{\p{C+C_\epsilon|\lambda|^{1-\alpha+\epsilon}}^2}{\lambda^2+\alpha^2}d\lambda\leq \frac{C_\epsilon}{\alpha}
\end{align*}
and, by lemma \ref{decay} and  $\epsilon<\frac{1}{2}$.
\begin{align*}
\int_{|\lambda|>2}|A_0(\alpha,\lambda)|\frac{\p{C+C_\epsilon|\lambda|^{1-\alpha+\epsilon}}^2}{\lambda^2+\alpha^2}d\lambda\leq C_\ep\int_{|\lambda|>2}|\lambda|^{-2-\alpha+2\ep}d\lambda \leq C_{\ep}
\end{align*}
Therefore, taking $\epsilon=\frac{1}{4}$ we have that
\begin{align*}
C_2^\epsilon(\alpha)<C_2,\quad C_3^\epsilon(\alpha)<\frac{C_3}{\alpha},
\end{align*}
where $C_2$ and $C_3$ are universal constants.

In addition for $C_1(\alpha)$ we have that, by applying lemma \ref{mainlemma2},
\begin{align*}
2\pi\alpha\overline{c}_\alpha+\int_{-\infty}^\infty A_0(\alpha,\lambda)\frac{\alpha^2}{\lambda^2+\alpha^2}d\lambda=\lim_{\ep\to 0^+}\int_{-\infty}^\infty A(\lambda,\alpha,\ep)\frac{\alpha^2}{\lambda^2+\alpha^2}d\lambda.
\end{align*}
And we can compute by Fubini
\begin{align*}
&\int_{-\infty}^{\infty}A(\lambda,\alpha,\ep)\frac{\alpha^2}{\alpha^2+\lambda^2}=-\overline{c}_\alpha \int_{0}^\infty x^{-2}\p{\frac{\sign(x-1)}{|x-1|^{\alpha}}+\frac{1}{|x-1|^\alpha}}\int_{-\infty}^\infty x^{i\lambda}\frac{\alpha^2}{\alpha^2+\lambda^2}d\lambda dx\\
&=-\alpha\overline{c}_\alpha\pi  \int_{0}^\infty e^{-\alpha|\log(x)|}x^{-2+\ep}\p{\frac{\sign(x-1)}{|x-1|^{\alpha}}+\frac{1}{|x-1|^\alpha}}dx\\
&=-\alpha\overline{c}_\alpha\pi\p{\int_{0}^1 x^{\alpha-2+\ep}\p{-\frac{1}{(1-x)^{\alpha}}+\frac{1}{(x+1)^\alpha}}dx+
\int_{1}^\infty x^{-\alpha-2+\ep}\p{\frac{1}{(x-1)^{\alpha}}+\frac{1}{(x+1)^\alpha}}dx}.
\end{align*}
Thus
\begin{align*}
&2\pi\alpha\overline{c}_\alpha+\int_{-\infty}^\infty A_0(\alpha,\lambda)\frac{\alpha^2}{\lambda^2+\alpha^2}d\lambda\\&=-\alpha\overline{c}_\alpha\pi\p{\int_{0}^1 x^{\alpha-2}\p{-\frac{1}{(1-x)^{\alpha}}+\frac{1}{(x+1)^\alpha}}dx+
\int_{1}^\infty x^{-\alpha-2}\p{\frac{1}{(x-1)^{\alpha}}+\frac{1}{(x+1)^\alpha}}dx}.
\end{align*}
This last expression is positive for $0<\alpha<1$. In order to check it we change variables to get
\begin{align*}
&2\pi\alpha\overline{c}_\alpha+\int_{-\infty}^\infty A_0(\alpha,\lambda)\frac{\alpha^2}{\lambda^2+\alpha^2}d\lambda\\&=-\alpha\overline{c}_\alpha\pi\p{\int_{0}^1 x^{\alpha-2}\p{-\frac{1}{(1-x)^{\alpha}}+\frac{1}{(x+1)^\alpha}}dx+\int_{0}^1 x^{2\alpha}\p{\frac{1}{(1-x)^{\alpha}}+\frac{1}{(x+1)^\alpha}}dx}
\end{align*}
The first integral in the last expression is equal to $\frac{2^{1-\alpha}}{-1+\alpha}$ and since $\alpha>0$ we have that the second one satifies
\begin{align*}
\int_{0}^1 x^{2\alpha}\p{\frac{1}{(1-x)^{\alpha}}+\frac{1}{(x+1)^\alpha}}dx<\int_{0}^1 \p{\frac{1}{(1-x)^{\alpha}}+\frac{1}{(x+1)^\alpha}}dx=\frac{2^{1-\alpha}}{1-\alpha}.
\end{align*}
Thus, we obtain $C_1(\alpha)>0$ for $0<\alpha<1$. It turns that

\begin{align*}
\frac{d}{dt}\int_{0}^\infty \frac{u(x,t)}{x^{1+\alpha}}dx \geq &C_1(\alpha)\p{\int_{0}^\infty \frac{u(x,t)}{x^{1+\alpha}}dx}^2-C_2||u_0||_{L^\infty}\int_{0}^\infty \frac{u(x,t)}{x^{1+\alpha}}dx-\frac{C_3||u_0||_{L^\infty}^2}{\alpha},
\end{align*}
Since $\int_{0}^\infty x^{-1-\alpha}u_0(x)dx$ can be chose arbitrarily large with respect to $||u_0||_{L^\infty}$, $\frac{1}{C_1(\alpha)}$, $\frac{1}{\alpha}$, $C_2$ and $C_3$ fixed $0<\alpha<1$ we can conclude that there exits $u_0$ such that $\int_{0}^\infty x^{-1-\alpha}u(x,t)dx$ blows up in finite time. This is a contradiction since $\int_{0}^\infty x^{-1-\alpha}u(x,t)dx\leq C ||u||_{C^1}$.

\begin{rem}\label{chamizo}The size of the constants $C_2$ and $C_3$ affect to the size of $\int_{0}^\infty x^{-1-\alpha}u_0(x)dx$ to be chosen in order to have blow up. The size of these constants can be reduced by a improvement of lemma \ref{hurwitz}. Although this reduction is not in the scope of this paper let us comment something about this fact. The function $$Z(i\lambda+\alpha)=(2\pi)^{-i\lambda-\alpha}\p{\zeta\p{i\lambda+\alpha,1+\frac{x}{2\pi}}-\zeta\p{i\lambda+\alpha,1-\frac{x}{2\pi}}},$$
where $\zeta(i\lambda+\alpha,x)$ ($x>0$) is the Hurwitz Zeta function (HZF). The HFZ is related with Riemann Zeta function (RZF) and one can use similar ideas used for the RZF in order to bound the HZF. We will give here some indications about how to get  decay for the HZF. All of these indications are due to Fernando Chamizo.
The bound $|\zeta(i\lambda+\alpha,x)|=O\p{|\lambda|^{1-\alpha}}$ can be reached by using Euler-Maclaurin summation formula (the constant in $O(|\lambda|^{1-\alpha})$  degenerates to a $log(|\lambda|)$ for $\alpha=1$).

By using the Poisson summation Formula one can get the estimate (3)-\cite{FF} which can be used to yields  the bound $|\zeta(i\lambda+\alpha,x)|\leq O(|\lambda|^{\frac{1-\alpha}{2}})$ for $\frac{1}{2}\leq \alpha \leq 1$ (the constant in $O(|\lambda|^{1-\alpha})$  degenerates to a $log(|\lambda|)$ for $\alpha=1$) .

By using Van der Corput type techniques one can get $|\zeta(\lambda,x)|\leq O(\log(|\lambda||\lambda|^{\frac{1-\alpha}{3}})$, for $\frac{1}{2}\leq \alpha\leq 1$.

\end{rem}

\textbf{Acknowledgements:}
The authors are supported by the Spanish Ministry of Economy under the ICMAT–Severo Ochoa grant SEV2015-0554 and the Europa Excelencia program ERC2018-092824. AC is partially supported by the MTM2017-89976-P and the ERC Advanced Grant 788250. The authors acknowledges Fernando Chamizo for his contributions in the remark  \ref{chamizo}.

\end{document}